\documentclass{amsart}

\usepackage{hyperref}
\usepackage{srcltx}
\usepackage{indentfirst}
\usepackage{epsfig}
\usepackage{psfrag}
\usepackage{graphicx}
\usepackage{overpic}
\usepackage{multirow}
\usepackage[mathscr]{eucal}
\usepackage[hang]{subfigure}
\usepackage{amsmath,amssymb, amsthm}
\usepackage[dvips]{color}
\usepackage[margin=1.4in]{geometry}

\newtheorem{theorem}{Theorem}
\newtheorem{lemma}{Lemma}

\newcommand\bbR{\mathbb{R}}
\newcommand\bbN{\mathbb{N}}
\newcommand\bbC{\mathbb{C}}
\newcommand\bxi{\boldsymbol{\xi}}
\newcommand\bx{\boldsymbol{x}}

\newcommand\bp{\boldsymbol{p}}
\newcommand\bX{\boldsymbol{X}}

\newcommand\bbf{\boldsymbol{f}}

\newcommand\bb{\boldsymbol{b}}

\newcommand\br{\boldsymbol{r}}

\newcommand\imag{{{\mathrm i}}}

\newcommand\bA{{\boldsymbol{A}}}
\newcommand\bB{{\boldsymbol{B}}}
\newcommand\bD{{\boldsymbol{D}}}
\newcommand\bE{{\boldsymbol{E}}}
\newcommand\bG{{\boldsymbol{G}}}
\newcommand\bg{\boldsymbol{g}}
\newcommand\bM{{\boldsymbol{M}}}
\newcommand\bR{{\boldsymbol{R}}}
\newcommand\bQ{\boldsymbol{Q}}
\newcommand\bJ{\boldsymbol{J}}

\newcommand\bLambda{\boldsymbol{\Lambda}}

\newcommand\dd{\,\mathrm{d}}
\newcommand\He{\mathit{He}}

\newcommand\mH{\mathcal{H}}

\newcommand\bbZ{\mathbb{Z}}
\newcommand\mE{\mathcal{E}}
\newcommand\norm[1]{\left\|{#1}\right\|_2}

\newcommand\bI{\boldsymbol{I}}

\newcommand\bbD{\mathbb{D}}
\newcommand\bH{{\boldsymbol{H}}}
\newcommand\hbbf{\hat{\bbf}}

\newcommand\pd[2]{\dfrac{\partial {#1}}{\partial {#2}}}

\newcommand\comment[1]{}

\graphicspath{{images/}}
{\theoremstyle{remark} \newtheorem{remark}{Remark}}
{\theoremstyle{corollary} \newtheorem{corollary}{Corollary}}

\makeatletter
\@addtoreset{equation}{section}
\makeatother

\DeclareMathOperator\re{Re}

\title[Suppression of Recurrence for Transport Equations]{Suppression
  of Recurrence in the Hermite-Spectral Method for Transport
  Equations}

\author{Zhenning Cai}
\address[Zhenning Cai]{Department of Mathematics, National University of Singapore,
  Level 4, Block S17, 10 Lower Kent Ridge Road, Singapore 119076}
\email{matcz@nus.edu.sg}

\author{Yanli Wang} \address[Yanli Wang]{School of Mathematics
  Science, Peking University, Beijing, China, 100871}
 \email{wylmath@pku.edu.cn}

\thanks{This work is supported by National University of Singapore Startup
Fund under Grant No. R-146-000-241-133. Yanli Wang is also supported
by the National Natural Scientific Foundation of China (Grant No.
11501042) and China Scholarship Council.}

\begin{document}
\begin{abstract}
  We study the unphysical recurrence phenomenon arising in the
  numerical simulation of the transport equations using
  Hermite-spectral method. From a mathematical point of view, the
  suppression of this numerical artifact with filters is theoretically
  analyzed for two types of transport equations. It is rigorously
  proven that all the non-constant modes are damped exponentially by
  the filters in both models, and formally shown that the filter does
  not affect the damping rate of the electric energy in the linear
  Landau damping problem. Numerical tests are performed to show the
  effect of the filters.

  \smallskip
  
\end{abstract}

\maketitle

\section{Introduction}

We consider a system with a large number of microscopic particles, and
the motion of these particles is governed by a force field. Instead of
the state of every individual particle, we are interested in the
collective behavior of these particles, such as the local density and
the mean velocity. To obtain such information, the system needs to be
properly modeled before carrying out the simulation.  Compared with
tracking the positions and velocities of all the particles as in the
method of molecular dynamics, a more efficient method is to use
the kinetic theory to describe the system in a statistical way. The basic
idea of the kinetic theory is to introduce a velocity distribution
function $f(\bx, \bxi, t)$, which denotes the number density of
particles in the position-velocity space, and the governing equation
for $f$ is
\begin{equation} 
  \pd{f}{t} + \bxi \cdot \nabla_{\bx} f
  + \bE \cdot \nabla_{\bxi} f = 0,
  \qquad t \in \bbR^+, \quad \bx \in \bbR^N,
    \quad \bxi \in \bbR^N,
\end{equation}
where $t$ denotes the time, $\bx$ denotes the position, and $\bxi$
stands for the velocity of the particles. The force field is given by
$\bE$. In this paper, we consider the one-dimensional case with
periodic boundary condition in space, and thus the equation for
$f(x,\xi,t)$ can be rewritten as
\begin{subequations} \label{eq:vlasov} 
  \begin{align}
    & \pd{f}{t} + \xi \pd{f}{x} + E \pd{f}{\xi} = 0,
      \qquad t \in \bbR^+, \quad x \in \bbR, \quad \xi \in \bbR, \\
    & \label{eq:periodic_bc} f(x,\xi,t) = f(x+D,\xi,t),
      \qquad \forall (x,\xi,t) \in \bbR \times \bbR \times \bbR^+,
  \end{align}
\end{subequations}
where $D$ is the period, and we assume that $E$ is also periodic and
independent of the velocity $\xi$, but may be a function of $t$ and
$x$. A typical example of this model is the Vlasov-Poisson (VP)
equation arising from the astrophysics and plasma physics, which
models the system formed by a large number of charged particles, and
the force is generated by a self-consistent electric field. Moreover,
Landau damping is one of the fundamental problems in the applications
of the VP equation. However, in the numerical simulations of Landau
damping, it is observed that an unphysical phenomenon called
``recurrence'' occurs for most grid-based solvers
\cite{Einkemmer2014Convergence}.

The recurrence is an unphysical periodic behavior in the numerical
solutions of the VP equation. It can be demonstrated by the simple
advection equation ($E = 0$ in \eqref{eq:vlasov}) whose exact solution
is $f(x,\xi,t) = f(x-\xi t, \xi, 0)$. It shows that any spatial wave
in the initial condition will cause a wave in the velocity domain in
the evolution of the solution, and the frequency of the wave gets
higher when $t$ increases. If the velocity domain is discretized by a
fixed grid, the exact solution cannot be well resolved when $t$ is
large. Particularly, the numerical solution may look smoother than the
exact solution and therefore appears similar to the solution at a
previous time. Such phenomenon has been reported in a number of
research works with different grid-based numerical methods
\cite{Filbet_1, Qiu2011positivity,Filbet2003,SL,CIP}. The appearance
of the recurrence may be postponed by using a larger number of
velocity grids \cite{zhou2001numerical,heath2012discontinuous}, which
also introduces larger computational cost. To avoid the recurrence,
the particle-in-cell method \cite{birdsall1991plasma,
hockeny1981computer, dawson1983particle} can be adopted and it is
reported in \cite{camporeale2016velocity} that the numerical result
does not present recurrence. However, since the particle-in-cell
method is a stochastic method, only half-order convergence can be
achieved. An idea of combining the two types of methods is introduced
in \cite{habbasi2011}, where the authors suppress the recurrence by
introducing some randomness into the grid-based methods.

In this paper, we consider another type of methods called transform
methods \cite{cheng, grant1967fourier}, where the distribution
function is mapped to the frequency space and the Fourier modes are
solved instead of the values on the grid points. Especially, we adopt
the Hermite-spectral method introduced by \cite{Holloway1996Spectral}
as the asymmetric Hermite method. The similar idea is adopted in
\cite{Wang} to get a slightly nonlinear discretization. For transform
methods, one can suppress recurrence by introducing filters to the
numerical methods \cite{parker2015fourier,Einkemmer2014strategy}, or
adding artificial collisions to the model
\cite{camporeale2016velocity, pezzi2016collisional}. The suppression
of the recurrence is numerically analyzed in \cite{Hilscher2013},
where it is shown that the collision has a damping effect for the
high-frequency modes, so that the distribution function is smoothed
out and the filamentation is weakened. However, a theoretical study of
its underlying mathematical mechanism is still missing in the
literature.

To fill the vacancy, we are going to conduct a theoretical analysis on
the relation between the filters and the recurrence. The analysis is
performed on two types of transport equations including the advection
equation and Vlasov-Poisson equation. For both types of equations, it
is shown by eigenvalue analysis that all the non-constant modes in the
discrete system converge to zero exponentially as the time goes to
infinity, and therefore the damping effect is rigorously proven.
Moreover, it is formally shown that the filter does not change the
damping rate of the electric energy in the case of linear Laudau
damping. Our numerical results are consistent with the analysis. In
the tests for linear Landau damping, numerical results with high
quality are observed with the filter introduced in \cite{HouLi2007}.

The rest of this paper is organized as follows. In Section
\ref{sec:Hermite}, we briefly introduce the Hermite-spectral method
and the filters. In Section \ref{sec:con} and \ref{sec:Vlasov}, two
types of equations are analyzed respectively. Some numerical
experiments are performed in Section \ref{sec:num}, and the concluding
remarks are given in Section \ref{sec:conclusion}.


\section{Hermite-spectral method and filtering} \label{sec:Hermite}
In this section, we focus on the velocity discretization of the Vlasov
equation. Following \cite{Wang}, we consider the
following approximation of the distribution function\footnote{In
  \cite{Wang}, the basis functions are translated
  and scaled in order to adapt the functions, while in this paper,
  such adaption is removed for easier analysis, and the resulting
  equations can be regarded as a linearized version of the model in
  \cite{Wang}.}:
\begin{equation}
  \label{eq:hermit_expansion}
  f(x, \xi, t) \approx \frac{1}{\sqrt{2\pi}}\sum_{i = 0}^M f_i(x, t)
  \He_i(\xi)\exp\left(-\frac{\xi^2}{2}\right),
\end{equation}
where $\He_i(\xi)$ is the normalized Hermite polynomials defined by
\begin{equation}
\He_n(\xi) = \frac{(-1)^n}{\sqrt{n!}}
  \exp\left( \frac{\xi^2}{2} \right) \frac{\mathrm{d}^n}{\mathrm{d}x^n}
  \exp \left( -\frac{\xi^2}{2} \right),
\end{equation}
and they have the following properties:
\begin{enumerate}
\item Orthogonality:
  \begin{equation}
    \label{eq:hermit_orthogonaliy}
    \frac{1}{\sqrt{2\pi}} \int_{\bbR}\He_m(\xi)\He_n(\xi)
      \exp\left(-\frac{\xi^2}{2}\right) \dd\xi =
    \delta_{mn}, \quad \forall  m, n\in \bbN;
  \end{equation}
\item Recursion relation:
  \begin{equation}
    \label{eq:hermit_recursion}
    \sqrt{n+1}\He_{n+1}(\xi)  = \xi \He_{n}(\xi) - \sqrt{n}
    \He_{n-1}(\xi), \quad \forall n\in \bbN;
  \end{equation}
\item Differential relation:
  \begin{equation}
    \label{eq:hermit_differential}
    \begin{gathered}
      \He_{n+1}'(\xi)  = \sqrt{n+1} \He_{n}(\xi), \\
      \frac{\mathrm{d}}{\mathrm{d}\xi}
        \left( \He_{n}(\xi)\exp\left(-\frac{\xi^2}{2}\right)\right)
      = -\sqrt{n+1} \He_{n+1}(\xi)\exp\left(-\frac{\xi^2}{2}\right),
      \quad \forall n \in \bbN.
    \end{gathered}
  \end{equation}
\end{enumerate}
Using these properties, the equations for the coefficients $f_i(x,t)$
can be derived by inserting \eqref{eq:hermit_expansion} into
\eqref{eq:vlasov} and integrating the result against $\He_k(\xi)$ with
$k = 0,\cdots,M$. By defining $\bbf = (f_0, f_1, \cdots, f_M)^T$, we
have the following evolution equations:
\begin{equation}
  \label{eq:moment_system}
  \pd{\bbf}{t} + {\bA} \pd{\bbf}{x}  - E\bB\bbf = 0, 
\end{equation}
where $\bA$ and $\bB$ are $(M+1)\times (M+1)$ matrices defined by
\begin{equation}
  \label{eq:matrixAB}
  \bA =
  \begin{pmatrix}
    0 &  1  & 0 & 0& \dots & 0\\
    1 & 0 & \sqrt{2} & 0& \dots & 0\\
    0 & \sqrt{2} & 0 & \sqrt{3}& \dots & 0\\
    0 & \ddots & \ddots & \ddots &\ddots & \vdots\\
    \vdots & \ddots & 0 & \sqrt{M-1} & 0 & \sqrt{M}\\
    0  & \dots & 0 & 0 & \sqrt{M} & 0
  \end{pmatrix}, \quad
  {\bB} =
  \begin{pmatrix}
    0 & 0 & 0 & \dots & 0\\
    1 & 0 & 0 & \dots & 0\\
    0 & \sqrt{2} & 0 & \dots & 0\\
    \vdots & \ddots & \ddots & \ddots & \vdots\\
    0  & \dots & 0 & \sqrt{M} & 0
  \end{pmatrix}.
\end{equation}

The system \eqref{eq:moment_system} is the semi-discrete transport
equation after spectral discretization of the velocity variable. As
will be seen below, such discretization suffers from a deficiency
called ``recurrence phenomenon'' \cite{Wang, parker2015fourier}, which
causes the non-physical echo of the electric energy when simulating the
plasma. In \cite{parker2015fourier}, the authors proposed the filtered
spectral method, and in the numerical results, the recurrence was
clearly suppressed. Here we adopt the similar method and apply the
filter after every time step. In detail, let $\bbf^n = (f_0^n, \cdots,
f_M^n)^T$ be the numerical solution of \eqref{eq:moment_system} at the
$n$th time step, and suppose a numerical scheme for
\eqref{eq:moment_system} is
\begin{equation}
  \label{eq:step_1}
  \bbf^{n+1} = \bQ(\bbf^{n}).
\end{equation}
When a filter is applied, the above scheme is altered as
\begin{equation}
  \label{eq:filter}
  \begin{aligned}
    & \bbf^{n, \ast} = \bQ(\bbf^n),  \\
    & f_i^{n+1} = \sigma_M(i) f_i^{n, \ast}, \qquad
      i = 0,\cdots,M,
  \end{aligned}
\end{equation}
where the filter $\sigma_M(i)$ satisfies
\begin{equation}
  \sigma_M(0) = 1, \qquad
  \lim_{M \rightarrow \infty} \sigma_M(i) = 1,
  \qquad \forall i \in \bbN.
\end{equation}
The filter is often interpreted as an operator with the effect of
diffusion \cite{Gottlieb2001spectral}. Especially, when we take the
exponential filter
\begin{equation} \label{eq:sigma_M}
\sigma_M(i) = \exp \left( -\alpha (i/M)^p \right),
\end{equation}
the method \eqref{eq:filter} is actually computing the solution to the
modified problem
\begin{equation} \label{eq:filtered_vlasov}
  \pd{f}{t} + \xi \pd{f}{x} + E \pd{f}{\xi} =
    -\alpha \frac{(-1)^p}{\Delta t \, M^{p}} \mathcal{D}^p f,
\end{equation}
where $\mathcal{D}$ is a linear operator defined by
\begin{equation}
  \mathcal{D}f(x,\xi,t) = \frac{\partial}{\partial \xi} \left[
    \exp \left( -\frac{\xi^2}{2} \right) \frac{\partial}{\partial \xi}
      \left( \exp \left( \frac{\xi^2}{2} \right) f(x,\xi,t) \right)
  \right].
\end{equation}
The time step $\Delta t$ needs to be chosen to ensure stability. Since
$\|\bB\|_2 = \sqrt{M}$ and the maximum eigenvalue of $\bA$ is the
maximum zero of $\He_{M+1}(\xi)$, which grows asymptotically as
$O(\sqrt{M})$, we choose the time step as $\Delta t \sim O(M^{-1/2})$
in this work. Thus one sees that when $M$ gets larger, the equation
formally converges to the transport equation \eqref{eq:vlasov} if $p
> 1/2$.

When the original transport equation \eqref{eq:vlasov} is replaced by
\eqref{eq:filtered_vlasov}, the semi-discrete system
\eqref{eq:moment_system} changes to
\begin{equation}
  \label{eq:moment_system_filter}
  \pd{\bbf}{t} + {\bA} \pd{\bbf}{x}  - E\bB\bbf = {\bH} \bbf,
\end{equation} 
where $\bH = -\Delta t^{-1} \alpha \, \mathrm{diag} \left\{ 0,
(1/M)^{p}, \cdots, [(M-1)/M]^{p}, 1 \right\}$. In general, we assume
that
\begin{equation} \label{eq:h}
  \bH = \mathrm{diag} \{ h_0, h_1, \cdots, h_M \}
\end{equation}
is an $(M+1)$ by $(M+1)$ diagonal matrix with non-positive diagonal
entries, and in order to keep the conservation of total number of
particles, we require that the first entry $h_0 = 0$. 

Below we are going to remove the spatial derivative by Fourier series
expansion. The periodic boundary condition \eqref{eq:periodic_bc}
shows that $\bbf$ is also periodic with respect to $x$. Thus we have
the following series expansions:
\begin{equation}
  \label{eq:fourier_expansion}
  \bbf = \sum_{m \in \bbZ} \hat{\bbf}^{(m)}\exp(\imag m k x), \quad
    E = \sum_{m \in \bbZ}\hat{E}^{(m)}\exp(\imag m k x),
    \qquad k = 2\pi / D,
\end{equation}
and Parseval's equality shows that
\begin{equation}
  \label{eq:fourier_equal}
  \|\bbf\|_2^2 = D\sum_{m \in \bbZ} \left\|\hat{\bbf}^{(m)}\right\|_2^2, 
  \qquad \|E\|_2^2  = D\sum_{m \in \bbZ} |\hat{E}^{(m)}|^2.
\end{equation}
Substituting \eqref{eq:fourier_expansion} into
\eqref{eq:moment_system_filter}, we get the equations for the Fourier
coefficients $\hat{\bbf}^{(m)}$:
\begin{equation}
  \label{eq:fourier_system_vlasov}
  \pd{\hat{\bbf}^{(m)}}{t} + \imag m k {\bA}\hat{\bbf}^{(m)} -
  \sum\limits_{l\in\bbZ}\hat{E}^{(l)}{\bB}\hat{\bbf}^{(m-l)} =
  {\bH} \hat{\bbf}^{(m)}, \quad m\in\bbZ.
\end{equation}
Based on the form \eqref{eq:fourier_system_vlasov}, we will show in
the following sections that the filter $\bH$ can suppress the
recurrence phenomenon, especially in the simulation of Landau damping.

\section{Advection Equation} \label{sec:con}
\subsection{Recurrence without filter}
\label{sec:nofilter}
We will begin our discussion with a simple case $E = 0$, and thus the
transport equation \eqref{eq:vlasov} becomes
\begin{equation}
  \label{eq:transport}
  \pd{f}{t} + \xi \pd{f}{x} = 0,
    \qquad t \in \bbR^+, \quad x \in \bbD, \quad \xi \in \bbR,
\end{equation}
where $\bbD = [0, D]$ and the periodic boundary condition is imposed.
It is known that the recurrence can be observed in such a simple
advection equation with initial value
\begin{equation} \label{eq:init}
  f(x,\xi,0) = \frac{1}{\sqrt{2\pi}} (1 + \epsilon \cos(kx))
    \exp \left( -\frac{\xi^2}{2} \right), \qquad k = \frac{2\pi}{D},
\end{equation}
and the recurrence time can be exactly given if the velocity is
discretized with a uniform grid \cite{Pohn2005Eulerian,
  habbasi2011}. For Hermite-spectral method, the system
\eqref{eq:fourier_system_vlasov} is correspondingly reduced to
\begin{equation}
  \label{eq:fourier_system_te}
  \pd{\hat{\bbf}^{(m)}}{t} + \imag m k \bA \hat{\bbf}^{(m)} =
    {\bH}  \hat{\bbf}^{(m)}, \quad m\in\bbZ.
\end{equation}
If no filter is applied, $\bH = 0$, and then the solution to the above
system is
\begin{equation}
  \label{eq:sol_no_filter}
  \hat{\bbf}^{(m)}(t)
    = \exp\left(-\imag m k t {\bA} \right)\hat{\bbf}^{(m)}(0)
    = {\bR^T}\exp\left(-\imag m k t {\bLambda} \right)
      {\bR}\hat{\bbf}^{(m)}(0),
\end{equation}
where $\bR$ is an $(M+1) \times (M+1)$ orthogonal matrix satisfying
${\bR} {\bLambda} {\bR}^T = {\bA}$, and $\bLambda$ is an $(M+1)\times
(M+1)$ real diagonal matrix due to the symmetry of the real matrix
$\bA$. The equality \eqref{eq:sol_no_filter} indicates that 
\begin{equation} \label{eq:l2_norm}
\norm{\hat{\bbf}^{(m)}(t)}^2 = \norm{\hat{\bbf}^{(m)}(0)}^2,
  \qquad \forall m \in \mathbb{Z}, \quad t \in \bbR^+.
\end{equation}

To observe the recurrence, we suppose that the exact solution to 
\eqref{eq:transport} and \eqref{eq:init} can be written as
\begin{equation}
f(x,\xi,t) = \frac{1}{\sqrt{2\pi}} \sum_{i=0}^M \sum_{m \in \bbZ}
  \hat{f}_{\mathrm{ex},i}^{(m)}(t) \exp(\imag mkx)
  \He_i(\xi) \exp \left( -\frac{\xi^2}{2} \right).
\end{equation}
By straightforward calculation, we have
\begin{equation} \label{eq:exact}
\hat{f}_{\mathrm{ex},i}^{(m)}(t) = \left\{ \begin{array}{ll}
  \delta_{i0}, & \text{if } m = 0, \\
  \dfrac{\epsilon}{2} \dfrac{(\imag mkt)^i}{\sqrt{i!}}
    \exp \left( -\dfrac{k^2 t^2}{2} \right), & \text{if } m = \pm 1, \\
  0, & \text{otherwise}.
\end{array} \right.
\end{equation}
Therefore
\begin{equation} \label{eq:quad_decay}
\sum_{i=0}^M |\hat{f}_{\mathrm{ex},i}^{(\pm 1)}(t)|^2 =
  \frac{\epsilon^2}{4} \sum_{i=0}^M \frac{(kt)^{2i}}{i!} \exp(-k^2 t^2),
\end{equation}
which decays to zero as $t \rightarrow \infty$. The relation
\eqref{eq:l2_norm} shows that this property is not maintained after
discretization. In fact, if $(2\pi)^{-1} k t \bLambda$ is close to an
integer matrix for some $t$, then $\hat{\bbf}^{(m)}(t)$ is close to
$\hat{\bbf}^{(m)}(0)$ for all $m \in \bbZ$, which turns out to be a
``recurrence''. Some illustration will be given in Section
\ref{sec:num}.

\subsection{Suppression of recurrence with filter} \label{sec:supp}
When a filter is applied, the solution to \eqref{eq:fourier_system_te}
is 
\begin{equation}
  \label{eq:fourier_system_sol2}
  \hat{\bbf}^{(m)}(t) = \exp\Big(\left(-\imag m k {\bA} + \bH\right) t
  \Big)\hat{\bbf}^{(m)}(0).
\end{equation}
In this subsection, we assume that $\bH$ is a filter whose first
diagonal entry is zero and last diagonal entry is nonzero. Actually,
almost all filters have such a form so that the high-frequency modes
can be damped while the low-frequency modes are not disturbed. Based
on this assumption, we have the following theorem:
\begin{theorem}
  \label{thm:transport_1}
  Let $\bA_m = -\imag m k {\bA} + {\bH}$. Then for all $m \in \bbZ
  \backslash \{0\}$, all the eigenvalues of $\bA_m$ have negative real
  parts.
\end{theorem}
The above theorem shows that for all $m \in \bbZ \backslash \{0\}$,
$\hat{\bbf}^{(m)}(t) \rightarrow 0$ as $t \rightarrow +\infty$, which
fixes the undesired property \eqref{eq:l2_norm}, and agrees with the
decaying behavior of the exact solution to the advection equation.
Thus the recurrence is suppressed. The proof of the theorem requires
the following lemma:

\begin{lemma}
  \label{lem:root_hermite}
  Let $\bM = (a_{ij})_{N\times N}$ be a symmetric tridiagonal matrix
  with nonzero subdiagonal entries. Define $p_n(\lambda)$, $n =
  1,\cdots,N$ as the characteristic polynomial of the $n$th order
  leading principle submatrix of $\bM$. Then the following statements
  hold:
  \begin{itemize}
    \item The roots of $p_n$ and $p_{n+1}$ are interlacing;
    \item If $\lambda$ is an eigenvalue of $\bM$, then the associated
      eigenvector is $\br = (r_1, \cdots, r_N)^T$ with $r_1 = 1$ and
      $r_n = (-1)^{n-1} p_{n-1}(\lambda) / (a_{12} a_{23} \cdots
      a_{n-1,n})$ for $n > 1$;
    \item If $\br = (r_1, \cdots, r_N)^T$ is an eigenvector of $\bM$,
      then $r_N \neq 0$.
  \end{itemize}
\end{lemma}
In the above lemma, the first statement is a well-known result for the
interlacing system. We refer the readers to
\cite{hamming1989introduction} for the proof. The second result can be
found in \cite{Wilkinson1965} and it can also be checked directly
using the definition of eigenvectors. The third statement is obviously
a result of the first two statements.

Now we give the proof of Theorem \ref{thm:transport_1} as below:
{\renewcommand\proofname{Proof of Theorem \ref{thm:transport_1}}
\begin{proof}
  Let $\lambda$ be an eigenvalue of $\bA_m$ for some $m \in \bbZ
  \backslash \{0\}$. We first show that $\re \lambda \leqslant 0$.
  Suppose $\br$ is the associated eigenvector. Using ${\bA_m}\br =
  \lambda \br$, we have that
  \begin{equation}
    0 \geqslant 2 \br^{\ast}\bH \br =
    \br^{\ast} \left(\bA_m^{\ast}  + \bA_m\right) \br
    = (\lambda + \bar{\lambda}) \br^{\ast}\br
    = 2 \|\br\|^2 \re \lambda.
  \end{equation}
  Thus $\re \lambda \leqslant 0$.

  It remains only to show that $\re \lambda \neq 0$. If there exist
  $\lambda_I \in \bbR$ and $\br \in \bbC^{M+1}$, such that ${\bA_m}\br
  = \imag \lambda_I \br$, then
  \begin{equation}
    \imag \lambda_I \|\br\|^2 = \br^{\ast} \bA_m \br
      = -\imag mk \br^* \bA \br + \br^* \bH \br.
  \end{equation}
  The symmetry of both $\bA$ and $\bH$ yields $\br^* \bH \br = 0$,
  which is equivalent to $\bH \br = 0$ since $\bH$ is diagonal.
  Recalling that the last diagonal entry of $\bH$ is assumed to be
  strictly negative, we know that the last component of $\br$ is zero.
  Furthermore, we have
  \begin{equation}
    \bA \br = \frac{\imag}{mk} (\bA_m - \bH) \br
      = -\frac{\lambda_I}{mk} \br,
  \end{equation}
  which indicates that $\br$ is an eigenvector of $\bA$. According
  to Lemma \ref{lem:root_hermite}, the last component of $\br$ must be
  nonzero, which is a contradiction. Therefore $\bA_m$ does not have
  purely imaginary eigenvalues, which concludes the proof.
\end{proof}}

The following theorem shows that the convergence rate has a lower
bound for all non-constant Fourier modes:
\begin{theorem}
  \label{thm:jordan_normal}
  For all $m \in \bbZ \backslash \{0\}$, suppose
  \begin{equation}
    \label{eq:jordan_norm}
    \bA_m := -\imag m k \bA +  \bH = \bR_m \bJ_m \bR_m^{-1},
  \end{equation}
  where $\bJ_m$ is the Jordan normal form of $\bA_m$, and every column
  of $\bR_m$ is a unit vector. Then there exists a constant $C^{(0)} >
  0$, such that 
  \begin{equation}
    \label{eq:bound_R}
    \norm{\bR_m} \leqslant C^{(0)}, \quad \norm{\bR_m^{-1}} \leqslant C^{(0)}.
  \end{equation}
  And there exists a constant $\lambda^{(0)} > 0$, such that for any
  $m \in \bbZ \backslash \{0\}$, all the eigenvalues of $\bA_m$ have
  real parts less than $-\lambda^{(0)}$.
\end{theorem}
\begin{proof}
  Consider the matrix
  \begin{equation} \label{eq:B_m}
    \bB_m := \frac{\imag}{mk} \bA_m = \bA + \frac{\imag}{mk} \bH,
    \qquad m \in \bbZ \backslash \{0\}.
  \end{equation}
  Apparently the characteristic polynomial of $\bB_m$ converges to the
  characteristic polynomial of $\bA$ as $m \rightarrow \infty$.
  Thus all the eigenvalues of $\bB_m$ also converge to the
  eigenvalues of $\bA$ as $m \rightarrow \infty$. Lemma
  \ref{lem:root_hermite} implies that all the eigenvalues of $\bA$ are
  distinct. Therefore there exists an $m_0 > 0$ such that all the
  eigenvalues of $\bB_m$ are distinct if $|m| > m_0$. The relation
  between $\bA_m$ and $\bB_m$ indicates that all the eigenvalues of
  $\bA_m$ are also distinct when $|m| > m_0$. In this case, the
  similarity transformation \eqref{eq:jordan_norm} becomes a
  diagonalization of $\bA_m$, and every column of $\bR_m$ is a unit
  eigenvector of $\bA_m$ (or $\bB_m$).

  To show the bound \eqref{eq:bound_R}, it is sufficient to show that
  $\|\bR_m\|_2$ and $\|\bR_m^{-1}\|_2$ have a uniform upper bound for
  all $|m| > m_0$. Let $\br_i^{(m)}$ be the $i$th column of $\bR_m$,
  and suppose $\bB_m \br_i^{(m)} = \mu_i^{(m)} \br_i^{(m)}$. Since
  exchanging two columns of $\bR_m$ does not change the norms
  $\|\bR_m\|_2$ and $\|\bR_m^{-1}\|_2$, we can assume
  \begin{equation}
    \lim_{m \rightarrow \infty} \mu_i^{(m)} = \lambda_i,
  \end{equation}
  where $\lambda_i$ is the $i$th eigenvalue of $\bA$. Using the fact
  that $\br_i^{(m)}$ is a unit vector, we have that
  \begin{equation*}
    0 = \lim_{m\rightarrow \infty}
        \left( \bB_m - \mu_i^{(m)} \bI \right) \br_i^{(m)}
      = \lim_{m\rightarrow \infty} (\bA  - \lambda_i \bI) \br_i^{(m)}.
  \end{equation*}
  Hence,
  \begin{equation}
    \label{eq:lim_eigenvector}
    \lim_{m\rightarrow \infty} \br_i^{(m)} = \br_i,
  \end{equation}
  where $\br_i$ is the unit eigenvector of $\bA$ associated with the
  eigenvalue $\lambda_i$. Thus $\bR_m$ has a limit $\bR$ as $m
  \rightarrow \infty$, and the limit diagonalizes $\bA$ as $\bA = \bR
  \bLambda \bR^{-1}$, which naturally leads to the bound
  \eqref{eq:bound_R}.

  To show that the bound $\lambda^{(0)}$ exists, we use the unity of
  vectors $\br_i^{(m)}$ to get
  \begin{equation}
      \lim_{m\rightarrow \infty} \re \lambda_i^{(m)} 
      =\lim_{m\rightarrow \infty} \re
        \left( \left( \br_i^{(m)} \right)^{\ast}\bA_m\br_i^{(m)} \right)
      = \lim_{m\rightarrow \infty} 
        \left( \br_i^{(m)} \right)^{\ast} \bH \br_i^{(m)}
      = \br_i^{\ast} \bH \br_i < 0.
  \end{equation}
  The last inequality comes from the proof of Theorem
  \ref{thm:transport_1}. The existence of the negative limit shows the
  existence of the negative upper bound.
\end{proof}

The above theorem gives an upper bound for the real parts of the
eigenvalues. A direct corollary is
\begin{corollary}
  \label{lem:bound_expm}
  $\forall m \in \bbZ$ and $t > 0$, $\left\|\exp(t \bA_m)\right\|_2$
  is uniformly bounded, and if $m\neq 0$, it holds that
  \begin{equation}
    \label{eq:expAm1}
    \|\exp\left(t \bA_m\right)\|_2 \leqslant
      C^{(1)} \left(t^M + 1\right)
      \exp\left(-\lambda^{(0)}t\right),
  \end{equation}
  where $C^{(1)} = (M+1) \left( C^{(0)} \right)^2$, and the constants
  $\lambda^{(0)}$ and $C^{(0)}$ are introduced in Theorem
  \ref{thm:jordan_normal}.
\end{corollary}
The estimate \eqref{eq:expAm1} is a result of Theorem
\ref{thm:jordan_normal} and the following lemma:
\begin{lemma}
For any matrix $\bM$, suppose $\bJ$ is its Jordan normal form and $\bM
= \bX \bJ \bX^{-1}$. The following estimate holds for the norm of
$\exp(t \bM)$:
\begin{equation}
\|\exp(t \bM)\|_2 \leqslant \beta \|\bX\|_2 \|\bX^{-1}\|_2
  \max_{0\leqslant i \leqslant \beta-1}
    \frac{t^i}{i!} \mathrm{e}^{-\alpha t},
\end{equation}
where $\beta$ is the maximum dimension of the Jordan blocks, and
$\alpha$ is the maximum real part of the eigenvalues of $\bM$.
\end{lemma}
The lemma can be found in \cite{Kagstrom1977bounds}. The uniform
boundedness in Corollary \ref{lem:bound_expm} is an immediate result
of \eqref{eq:expAm1} and $\|\exp(t\bA_0)\|_2 = \|\exp(t \bH)\|_2 = 1$.

The estimate \eqref{eq:expAm1} shows the linearly exponential decay
of the non-constant Fourier modes. Compared with
\eqref{eq:quad_decay}, the decay rate is still not fast enough, which
indicates that the recurrence is not fully removed. However, when $M$
is sufficiently large, the filtered spectral method can give accurate
approximation of the decay rate up to some time $T$, and after time
$T$, the values of both the exact solution and the numerical solution
are already small enough, and therefore the numerical result can still
be considered as accurate, although the decay rate may not be exact.
Examples will be given in Section \ref{sec:num} to show the
aforementioned behavior.

\subsection{Advection Equation with an Exponentially Decaying Force}
\label{sec:force}
The above result can be extended to the case with a given decaying
force field. Here we assume that
\begin{equation}
  \label{eq:electric_force}
  E(x, t) = \exp(-\alpha(t))w(x, t), \quad x \in \bbD, 
\end{equation}
where $\alpha(t) \geqslant \alpha_{E} t > 0$ and
$w(x,t) \in L^{\infty}(\bbD\times [0, +\infty))$. Such a force field
mimics the electric force field in the Vlasov-Poisson equations, where
the self-consistent force decays exponentially. Therefore, it can be
expected that the behavior of this equation is similar to the linear
Landau damping. Again, we study the equations of Fourier coefficients
\eqref{eq:fourier_system_vlasov} instead of the original system
\eqref{eq:moment_system_filter}. It will be shown that the system
\eqref{eq:moment_system_filter} has a steady state solution in which
only the constant modes are nonzero. Here we only consider the case
with filter, which means that the last diagonal entry of $\bH$ is
strictly negative.

For the purely advective equation, each $\hbbf^{(m)}$ can be
considered independently, while in this case, the Fourier coefficients
$\hbbf^{(m)}$ are fully coupled for all $m \in \bbZ$. Therefore, we
define a Hilbert space $\mathscr{H}$ whose elements have the form
\begin{equation}
  \label{eq:vector}
  \hat{\bg} = \left(
    \cdots, \hat{\bg}^{(-m)}, \cdots, \hat{\bg}^{(-1)},
    \hat{\bg}^{(0)}, \hat{\bg}^{(1)}, \cdots, \hat{\bg}^{(m)}, \cdots
  \right),
\end{equation}
where each $\hat{\bg}^{(m)}$ is a vector in $\bbC^{M+1}$. For any
vector in $\mathscr{H}$, below we always use the superscript ``$(m)$''
to denote its $m$th component as in \eqref{eq:vector}. The inner
product of $\mathscr{H}$ is defined as
\begin{equation}
\langle \hat{\bg}_1, \hat{\bg}_2 \rangle =
  \sum_{m \in \bbZ} \big( \hat{\bg}_1^{(m)} \big)^* \hat{\bg}_2^{(m)},
  \qquad \forall \hat{\bg}_1, \hat{\bg}_2 \in \mathscr{H}.
\end{equation}
Therefore by Parseval's inequality \eqref{eq:fourier_equal},
$\mathscr{H}$ corresponds to the space $[L^2(\bbD)]^{M+1}$ in the
original space. For the purpose of uniformity, the norm in the Hilbert
space $\mathscr{H}$ is denoted as $\|\cdot\|_2$ below.

To represent the system \eqref{eq:moment_system_filter}, we introduce
two operators $\hat{\bA}$ and $\hat{\bB}(t)$, which are
transformations on $\mathscr{H}$, defined by
\begin{equation}
  \begin{array}{rcl}
    \tilde{\bg} = \hat{\bA} \bg & \Longleftrightarrow &
      \tilde{\bg}^{(m)} = \bA_m \bg^{(m)}, \: \forall m \in \bbZ, \\[5pt]
    \tilde{\bg} = \hat{\bB}(t) \bg & \Longleftrightarrow &
      \displaystyle \tilde{\bg}^{(m)} =
        \sum_{l \in \bbZ} \hat{w}^{(l)}(t) \bB \bg^{(m-l)},
      \: \forall m \in \bbZ.
  \end{array}
\end{equation}
where the matrix $\bA_m$ is defined in Theorem \ref{thm:transport_1}
as $\bA_m = -\imag m k {\bA} + {\bH}$, and $\hat{w}^{(l)}(t)$ are the
Fourier coefficients for $w(x,t)$:
\begin{equation}
  \label{eq:fourier_w}
  w(x, t) = \sum_{m\in \bbZ} \hat{w}^{(m)}(t) \exp(\imag k x),
    \quad m \in \bbZ.
\end{equation}
Thus \eqref{eq:moment_system_filter} can be written as
\begin{equation}
  \label{eq:differential_form}
  \frac{\partial \hbbf(t)}{\partial t} =
    \hat{\bA} \hbbf(t) + \exp(-\alpha(t)) \hat{\bB}(t) \hbbf(t).
\end{equation}
Here $\hbbf(\cdot)$ is a map from $\bbR^+$ to $\mathscr{H}$. Applying
Duhamel's principle, we can obtain its integral form as
\begin{equation}
  \label{eq:integral_form}
  \hat{\bbf}(t) = \exp\left(t\hat{\bA}\right)\hat{\bbf}(0) +
  \int_0^t \exp(-\alpha(s)) \exp\left((t-s)\hat{\bA}\right)
    \hat{\bB}(s) \hat{\bbf}(s)\dd s.
\end{equation}
For the operators $\exp(t \hat{\bA})$ and $\hat{\bB}(t)$, we claim that
\begin{lemma}
  \label{lem:bounds}
  For all $t \geqslant 0$, the operators $\exp(t \hat{\bA})$ and
  $\hat{\bB}(t)$ are uniformly bounded, i.e. there exist constants $C_A$
  and $C_B$ such that
  \begin{equation}
    \label{eq:bounds}
    \left\| \exp(t \hat{\bA}) \right\|_2 \leqslant C_A, \quad
    \left\| \hat{\bB}(t) \right\|_2 \leqslant C_B, \qquad
      \forall t > 0.
  \end{equation}
\end{lemma}
This lemma can be easily obtained by Corollary \ref{lem:bound_expm}
and the boundedness of $w(x,t)$. The detailed proof is left for the
readers. Below we state the main result of this subsection:
\begin{theorem}
  \label{thm:steadysol_vlasov}
  Let $\hbbf(t)$ be the solution to \eqref{eq:integral_form}. There
  exists $\hbbf_{\infty} \in \mathscr{H}$ such that
  \begin{equation}
    \label{eq:limit}
    \lim_{t\rightarrow +\infty} \|\hbbf(t) - \hbbf_{\infty}\|_2 = 0,
  \end{equation}
  and $\hbbf_{\infty}$ satisfies
  \begin{equation}
    \label{eq:f_steady}
    \hbbf_{\infty}^{(m)} = 0, \qquad \forall m \in \bbZ \backslash \{0\}.
  \end{equation}
\end{theorem}

This theorem shows that the steady solution of
\eqref{eq:moment_system_filter} exists and is a constant. It is known
from Section \ref{sec:nofilter} that when recurrence
appears, some non-constant modes will never decay, which causes the
phenomenon that similar solutions appear again and again when all
these modes are close to the peaks of the waves. Consequently, no
steady state solution exists in such a case. In this sense, the
theorem implies the suppression of recurrence by the filter. Before
proving this theorem, we first show the boundedness of the solution:
\begin{lemma}
  \label{lem:exist_vlasov}
  Let $\hbbf$ be the solution to \eqref{eq:integral_form}. Then there
  exists a constant $C_f$ such that
  \begin{equation}
    \label{eq:l2_vlasov_force}
    \norm{\hbbf(t)} \leqslant C_f \norm{\hbbf(0)}.
  \end{equation}
\end{lemma}

\begin{proof}
  Taking the norm on both sides of \eqref{eq:integral_form} and
  plugging in the inequalities \eqref{eq:bounds}, we obtain
  \begin{equation}
    \begin{split}
    \norm{\hbbf(t)} & \leqslant C_A\norm{\hbbf(0)} +
      C_A\int_0^t \exp(-\alpha(s)) \norm{\hat{\bB}\hbbf(s)} \dd s \\
    & \leqslant C_A \norm{\hbbf(0)} +
      C_A C_B \int_0^t \exp(-\alpha_E s) \norm{\hbbf(s)} \dd s
    \end{split}
  \end{equation}
  By Gronwall's inequality, we immediately have the estimation
  \eqref{eq:l2_vlasov_force} with the constant $C_f = C_A
  \exp(\|w\|_{\infty} C_A C_B/ \alpha_E)$.
\end{proof}

Now we present the proof of Theorem \ref{thm:steadysol_vlasov}:
{\renewcommand\proofname{Proof of Theorem \ref{thm:steadysol_vlasov}}
\begin{proof}
  Let $\hat{\bb}(t) = \hat{\bB}(t) \hbbf(t)$. The boundedness of
  $\hat{\bB}(t)$ and Lemma \ref{lem:exist_vlasov} yield that
  \begin{equation}
    \label{eq:b}
    \|\hat{\bb}(t)\|_2 \leqslant C_B C_f \|\hbbf(0)\|_2.
  \end{equation}
  Writing the integral system \eqref{eq:integral_form} as
  \begin{equation} \label{eq:fm_integral_form}
    \hat{\bbf}^{(m)}(t) = \exp(t\bA_m) \hat{\bbf}^{(m)}(0) +
    \int_0^t \exp(-\alpha(s)) \exp((t-s) \bA_m)
      \hat{\bb}^{(m)} (s) \dd s, \quad \forall m \in \bbZ,
  \end{equation}
  we can bound $\left\| \hbbf^{(m)}(t) \right\|_2$ with $m \in \bbZ
  \backslash \{0\}$ by
  {\small \begin{equation} \label{eq:nonzero_m}
    \begin{split}
      & \left( \sum_{m \in \bbZ \backslash \{0\}}
          \left\| \hbbf^{(m)}(t) \right\|_2^2 \right)^{1/2}
	\leqslant C^{(1)} (t^M+1) \exp(-\lambda^{(0)}t)
          \left( \sum_{m \in \bbZ \backslash \{0\}}
            \left\| \hbbf^{(m)}(0) \right\|_2^2
          \right)^{1/2} \\
      & \qquad + \int_0^t C^{(1)} [(t-s)^M+1]
        \exp \left( -\lambda^{(0)}(t-s) \right) \exp(-\alpha_E s)
          \left( \sum_{m \in \bbZ \backslash \{0\}}
            \left\| \hat{\bb}^{(m)}(s) \right\|_2^2 \right)^{1/2} \dd s \\
      & \quad \leqslant C^{(1)} (t^M+1) \exp(-\lambda^{(0)}t)
      \left( \|\hbbf(0)\|_2 +
        \int_0^t \exp \left( (\lambda^{(0)} - \alpha_E) s \right)
        \|\hat{\bb}(s)\|_2 \dd s
      \right) \\
      & \quad \leqslant C^{(1)} (t^M+1) \left(
        \exp(-\lambda^{(0)}t) + 
        \frac{C_B C_f}{|\lambda^{(0)} - \alpha_E|}
          \left[ \exp(-\lambda^{(0)}t) + \exp(-\alpha_E t) \right]
      \right).
    \end{split}
  \end{equation}
}
  The right hand side goes to zero as $t$ goes to infinity. Thus it
  remains only to prove $\hbbf^{(0)}(t)$ has a limit.
  
  When $m = 0$, the equation \eqref{eq:fm_integral_form} becomes
  \begin{equation}
    \hat{\bbf}^{(0)}(t) = \exp(t\bH) \hat{\bbf}^{(0)}(0) +
    \int_0^t \exp(-\alpha(s)) \exp((t-s) \bH)
      \hat{\bb}^{(0)} (s) \dd s.
  \end{equation}
  For any $i = 0,\cdots,M$, if $h_i < 0$, we can still get
  $\hat{f}_i^{(0)}(+\infty) = 0$ using the same method as in
  \eqref{eq:nonzero_m}. If $h_i = 0$,
  \begin{equation}
    \lim_{t\rightarrow +\infty} \hat{f}_i^{(0)}(t) = \hat{f}_i^{(0)}(0)
      + \int_0^{+\infty} \exp(-\alpha(s)) \hat{b}_i^{(0)} (s) \dd s.
  \end{equation}
  The uniform boundedness of $\hat{b}_i^{(0)}$ and $\alpha(t)
  \geqslant \alpha_E t > 0$ shows the convergence of the integral on
  the right hand side, which concludes the proof.
\end{proof}}

\begin{remark}
In Theorem \ref{thm:steadysol_vlasov}, one can also observe an
exponential convergence rate to the steady state solution. However,
the convergence rate depends on both the eigenvalue bound
$\lambda^{(0)}$ and the decay rate of the force field $\alpha_E$,
whereas in Section \ref{sec:supp}, the convergence rates depend only
on the eigenvalue bounds. In fact, if the force field does not decay,
the steady state solution may not exist. A simple example is to let
$E$ be a constant. Then the distribution function will keep moving in
the velocity space, and there is no mechanism which balances such a
force to achieve the steady state. It can also be expected that
current velocity discretization cannot well describe the system when
$t$ is large.
\end{remark}


\section{Linear Landau Damping in the Vlasov-Poisson Equation}
\label{sec:Vlasov}
In this section, we will generalize the result to the linear Landau
damping problem, for which the recurrence in the electric energy has
been widely observed \cite{pezzi2016collisional, Filbet_1}. To study
the linear Landau damping, we consider the dimensionless
Vlasov-Poisson (VP) equation which models the motion of a collection
of charged particles in the self-consistent electric field. For a
given distribution function $f(x,\xi,t)$, the self-consistent electric
field $E_{\mathrm{sc}}(x,t)$ is given by
\begin{equation} 
  \label{eq:normal_force} 
  \pd{E_{\mathrm{sc}}(x,t)}{x} =  \left(\int_{\bbR} f(x,\xi, t)\dd\xi
    -\frac{1}{D} \int_{\bbD\times\bbR} f(x,\xi, t)\dd\xi \dd x\right).
\end{equation}
with constraint
\begin{equation} \label{eq:zero_mean}
  \int_{\bbD} E_{\mathrm{sc}}(x,t) \dd x = 0.
\end{equation}
Thus the Vlasov-Poisson equation can be written as \eqref{eq:vlasov}
with $E(x,t) = E_{\mathrm{sc}}(x,t)$. The initial condition is a
uniform Gaussian distribution with perturbed electric charge density
in the $x$-space:
\begin{equation}
  \label{eq:initial_vlasov}
  f(x, \xi, 0) = \frac{1}{\sqrt{2\pi}}
    [1 + \epsilon \exp(\imag kx)] \exp\left(-\frac{\xi^2}{2}\right).
\end{equation}
When $\epsilon$ is small, it is known that the total electric energy
\begin{equation} \label{eq:electric_energy}
  \mE(t) = \sqrt{\int_{\bbD} |E_{\mathrm{sc}}(x,t)|^2 \dd x}
\end{equation}
decays exponentially. For this example, we say that a recurrence
exists in some numerical methods if the electric energy does not decay
with time.

\subsection{Asymptotic expansion and recurrence}
Following the classical analysis \cite{Goldston2016Introduction}, we
expand the coefficients $\bbf = (f_0, f_1, \cdots, f_M)^T$ in terms of
$\epsilon$:
\begin{equation}
  \label{eq:ce_f}
  \bbf(x, t) = \hat{\bbf}^{(0)}(t)
    + \epsilon \hat{\bbf}^{(1)}(t)\exp(\imag k x)
    + \epsilon^2 \hat{\bbf}^{(2)}(t)\exp(2\imag k x) + \cdots,
\end{equation}
where
$\hat{\bbf}^{(m)} = (\hat{f}_0^{(m)}, \hat{f}_1^{(m)}, \cdots,
\hat{f}_M^{(m)})^T$.
Correspondingly, the electric field $E(x, t)$ is expanded as
\begin{equation}
  \label{eq:ce_E}
  E(x, t) = \epsilon \hat{E}^{(1)}(t)\exp(\imag k x)
    + \epsilon^2 \hat{E}^{(2)}(t)\exp(2 \imag k x)
    + \epsilon^3 \hat{E}^{(3)}(t)\exp(3 \imag k x) + \cdots.
\end{equation}
Note that the definitions of $\hat{\bbf}^{(m)}$ and $\hat{E}^{(m)}$ are
slightly different from those defined in \eqref{eq:fourier_expansion}.
In this section, these symbols denote the Fourier coefficients scaled
by $\epsilon^m$. 
Inserting \eqref{eq:ce_f}\eqref{eq:ce_E} and
\eqref{eq:hermit_expansion} into \eqref{eq:normal_force}, we get that
\begin{equation}
  \label{eq:E_F}
  \hat{E}^{(m)} = -\frac{\imag}{mk} \hat{f}^{(m)}_0,
    \qquad \forall m = 1,2,\cdots.
\end{equation}

The equations for the coefficients of all orders can be obtained by
substituting \eqref{eq:ce_f} and \eqref{eq:ce_E} into
\eqref{eq:moment_system_filter} and matching the terms with the same
orders of $\epsilon$:
\begin{align*}
& O(1): & & \pd{{\hbbf}^{(0)}}{t} = \bH {\hbbf}^{(0)}, \quad
  \hbbf^{(0)}(0) = (1, 0,\cdots, 0)^T. \\
& O(\epsilon): & &
  \pd{{\hbbf}^{(1)}}{t} + \imag k \bA {\hbbf}^{(1)} +
  \frac{\imag}{k} \hat{f}^{(1)}_0{\bB} {\hbbf}^{(0)} = \bH {\hbbf}^{(1)},
  \quad \hbbf^{(1)}(0) = (1, 0,\cdots, 0)^T. \\
& O(\epsilon^m): & &
  \pd{{\hbbf}^{(m)}}{t} + \imag m k \bA {\hbbf}^{(m)} +
  \frac{\imag}{mk} \hat{f}^{(m)}_0{\bB} {\hbbf}^{(0)} +
  \frac{\imag}{k} \sum_{l=1}^{m-1} \frac{1}{l} \hat{f}^{(l)}_0{\bB} {\hbbf}^{(m-l)}
    = \bH {\hbbf}^{(m)}, \quad \hbbf^{(m)}(0) = 0.
\end{align*}
The last equation holds for all $m \geqslant 2$. Recalling that $\bH$ is a
diagonal matrix with its first entry being zero, we can easily obtain
from the $O(1)$ equation that $\hbbf^{(0)}(t) \equiv \hbbf^{(0)}(0) =
(1, 0,\cdots, 0)^T$. Therefore the other two equations can be
rewritten as
\begin{gather}
\label{eq:first_order}
  \pd{{\hbbf}^{(1)}}{t} +
    \imag k \left( \bA + \frac{1}{k^2} \bG \right){\hbbf}^{(1)}
  = \bH {\hbbf}^{(1)}, \qquad \hbbf^{(1)}(0) = (1, 0,\cdots, 0)^T; \\
\label{eq:high_order}
  \pd{{\hbbf}^{(m)}}{t} +
    \imag m k  \left( \bA + \frac{1}{(mk)^2} \bG \right){\hbbf}^{(m)} +
    \frac{\imag}{k} \sum_{l=1}^{m-1} \hat{f}^{(l)}_0{\bB} {\hbbf}^{(m-l)}
    = \bH {\hbbf}^{(m)}, \qquad \hbbf^{(m)}(0) = 0.
\end{gather}
Here we have used the symbol $\bG$ to denote the $(M+1)\times (M+1)$
matrix with only one nonzero entry locating at the second row and
first column:
\begin{equation}
\bG = \begin{pmatrix}
  0 & 0 & \cdots & 0 \\ 1 & 0 & \cdots & 0 \\
  0 & 0 & \cdots & 0 \\ \vdots & \vdots & \ddots& \vdots \\
  0 & 0 & \cdots & 0
\end{pmatrix}.
\end{equation}

When $\bH = 0$, the recurrence phenomenon already exists in the
first-order equation \eqref{eq:first_order}. To observe this, we
introduce the following lemma:
\begin{lemma}
  \label{lem:matrixhatB}
  For any $m > 0$, there exists a diagonal matrix $\bD_m$ such that
  such that the matrix
  \begin{displaymath}
  \bD_m \left( \bA + \frac{1}{(mk)^2}\bG \right) \bD_m^{-1}
  \end{displaymath}
  is symmetric and tridiagonal, and therefore the matrix
  $\bA + (mk)^{-2}\bG$ is real diagonalizable.
\end{lemma}
\begin{proof}
  The matrix $\bD_m$ can be explicitly written as
  \begin{equation}
  \label{eq:D_m}
    \bD_m = \mathrm{diag} \left\{1, \frac{mk}{\sqrt{(mk)^2+1}},
      \frac{mk}{\sqrt{(mk)^2+1}}, \cdots,
      \frac{mk}{\sqrt{(mk)^2+1}} \right\}.
  \end{equation}
  It can then be easily verified that
  \begin{equation*}
    {\bD_m}\left(\bA + \frac{1}{(km)^2}\bG\right)\bD_m^{-1} = 
    \begin{pmatrix}
      0 & \sqrt{1+(mk)^{-2}} & 0 & 0& \dots & 0\\
      \sqrt{1+(mk)^{-2}} & 0 & \sqrt{2} & 0& \dots & 0\\
      0 & \sqrt{2} & 0 & \sqrt{3}& \dots & 0\\
      0 & \ddots & \ddots & \ddots &\ddots & \vdots\\
      \vdots & \ddots & 0 & \sqrt{M-1} & 0 & \sqrt{M}\\
      0  & \dots & 0 & 0 & \sqrt{M} & 0
    \end{pmatrix}.
  \end{equation*}
\end{proof}
The above lemma shows that when $\bH = 0$, all the eigenvalues of
$\imag k (\bA + k^{-2}\bG)$ are purely imaginary, which indicates that
$\hbbf^{(1)}$ does not decay as time increases. Consequently, the
relation \eqref{eq:E_F} shows that the electric energy will not decay,
resulting in the recurrence in the simple discretization of velocity
without using filters.

\subsection{Suppression of recurrence}
This section focuses on the effect of the filter. As in Section
\ref{sec:supp}, we assume again that the filter matrix $\bH$ is a
negative semidefinite diagonal matrix with its first diagonal entry
being zero and last diagonal entry being nonzero. In this section, we
are going to show an exponential decay of the solution, which
suppresses the recurrence. In detail, we have the following theorem:
\begin{theorem}
  \label{thm:vp}
  Let $\bbf(x,t)$ be the series \eqref{eq:ce_f} where $\hbbf^{(0)}(t)
  \equiv (1,0,\cdots,0)^T$ and $\hbbf^{(m)}$ with $m \geqslant 1$ are
  solved from the equations \eqref{eq:first_order}%
  \eqref{eq:high_order}. There exist constants $C^{(3)}$ and
  $\lambda^{(1)}$ such that
  \begin{equation}
    \label{eq:sol_vp}
    \norm{\bbf(\cdot,t) - \hbbf^{(0)}} \leqslant
      C^{(3)} \left(t^M+1 \right)\exp\left(-\lambda^{(1)}t \right),
  \end{equation}
  if $\epsilon$ is sufficiently small.
\end{theorem}

Similar to the case of the advection equation, the decay rate
$\lambda^{(1)}$ is associated with the eigenvalues of matrices
appearing in the equations. The following lemma gives a precise
definition of $\lambda^{(1)}$:
\begin{lemma}
  \label{lem:eig_hatA}
  For all $m \in \bbZ \backslash \{0\}$, define
  \begin{equation}
  \bG_m := -\imag m k \left( \bA + \frac{1}{(mk)^2}\bG \right) + \bH.
  \end{equation}
  Then there exists a uniform constant $\lambda^{(1)} > 0$, such that
  all the eigenvalues of $\bG_m$ have real parts less than
  $-\lambda^{(1)}$. Furthermore, there exists a constant $C^{(4)}$
  such that
  \begin{equation}
  \label{eq:exp_G}
    \| \exp(t \bG_m) \|_2 \leqslant C^{(4)}
      (t^M + 1) \exp \left( -\lambda^{(1)} t \right),
    \qquad \forall m \in \bbZ \backslash \{0\}.
  \end{equation}
\end{lemma}

\begin{proof}
Consider the matrix
\begin{equation}
\label{eq:DGDinv}
  \bD_m \bG_m \bD_m^{-1} = -\imag  m k \left(\bD_m
    \left( \bA + \frac{1}{(mk)^2}\bG \right) \bD_m^{-1}\right) + \bH,
\end{equation}
which has the same eigenvalues as $\bG_m$. Due to Lemma
\ref{lem:matrixhatB}, we see that the above matrix is symmetric and
tridiagonal, and all its subdiagonal entries are nonzero. Thus, we can
use the same strategy as in the proof of Theorem \ref{thm:transport_1}
to show that all the eigenvalues of \eqref{eq:DGDinv} have negative
real parts. Similarly, since
\begin{equation*}
   \lim\limits_{m\rightarrow \infty}\bD_m 
    \left( \bA + \frac{1}{(mk)^2}\bG \right) \bD_m^{-1} = \bA,
\end{equation*}
showing the existence of a uniform bound $\lambda^{(1)}$ is an analog
of the proof of Theorem \ref{thm:jordan_normal}. The details are left
to the readers.

To show the inequality \eqref{eq:exp_G}, we first follow the proof of
Corollary \ref{lem:bound_expm} to get
\begin{equation}
  \label{eq:exp_DGD}
  \left\|\exp(t \bD_m \bG_m \bD_m^{-1} )\right\|_2  \leqslant
    C^{(2)} \left(t^M + 1\right)\exp\left(-\lambda^{(1)}t\right),
  \quad \forall m \in \bbZ \backslash \{0\}.
\end{equation}
where $C^{(2)}$ is determined in the same way as $C^{(1)}$ in
Corollary \ref{lem:bound_expm} and Theorem \ref{thm:jordan_normal}.
Hence,
\begin{displaymath}
\begin{split}
  \|\exp(t \bG_m)\|_2 & \leqslant
    \|\bD_m^{-1}\|_2 \|\exp(t \bD_m \bG_m \bD_m^{-1})\|_2 \|\bD_m\|_2 \\
  &= \sqrt{1+\frac{1}{(mk)^2}}\|\exp(t \bD_m \bG_m \bD_m^{-1})\|_2
    \leqslant \sqrt{1+\frac{1}{k^2}} C^{(2)} (t^M + 1)
      \exp\left(-\lambda^{(1)}t\right).
\end{split}
\end{displaymath}
Thus \eqref{eq:exp_G} holds for $C^{(4)} = \sqrt{1+k^{-2}} C^{(2)}$.
\end{proof}

In the above lemma, the estimate \eqref{eq:exp_G} gives the basic
form of the decay rate. To turn \eqref{eq:exp_G} into an estimate of
the solution \eqref{eq:sol_vp}, we follow the steps below:
\begin{enumerate}
\item Show that each non-constant term in the series \eqref{eq:ce_f}
  decays exponentially.
\item Show the convergence of the series for small $\epsilon$.
\end{enumerate}
The first step is established by proving the following theorem:
\begin{theorem}
Let $\hbbf^{(m)}$, $m \geqslant 1$ be the solutions to the equations
\eqref{eq:first_order}\eqref{eq:high_order}. For each positive
integer $m$, there exists a constant $K^{(m)}$ such that
\begin{equation}
\label{eq:fm_est}
\left\| \hbbf^{(m)}(t) \right\|_2 \leqslant
  K^{(m)} (t^M + 1) \exp(-\lambda^{(1)}t).
\end{equation}
\end{theorem}

\begin{proof}
  Note that the equation \eqref{eq:first_order} is linear. Therefore
  we can take $m=1$ in \eqref{eq:exp_G} to get
  \begin{equation}
    \label{eq:norm_f1}
    \norm{\hbbf^{(1)}(t)} = \norm{\exp(\bG_1 t) \hbbf^{(1)}(0)}
    \leqslant C^{(4)}(t^M + 1)\exp\left(-\lambda^{(1)} t \right).
  \end{equation}
  Thus \eqref{eq:fm_est} is proven for $m = 1$ by taking $K^{(1)} =
  C^{(4)}$.

  For $m \geqslant 2$, we use induction and suppose that the
  inequality
  \begin{equation}
    \label{eq:norm_fj}
    \left\| {\hbbf}^{(j)}(t)\right\|_2 \leqslant {K}^{(j)}
      \left(t^M+1\right) \exp\left(-\lambda^{(1)} t\right), 
  \end{equation}
  holds for all $j < m$. To estimate $\hbbf^{(m)}(t)$, we write the
  explicit solution of \eqref{eq:high_order} as
  \begin{equation}
    \label{eq:sol_sec}
    {\hbbf}^{(m)}(t) = \int_0^t \exp((t-s) \bG_m) \bp_m(s) \dd s,
  \end{equation}
  where we have used the definition
  \begin{equation}
  \bp_m(t) := -\frac{\imag}{k} \sum_{l=1}^{m-1}
    \hat{f}^{(l)}_0(t) {\bB} {\hbbf}^{(m-l)}(t)
  \end{equation}
  for conciseness. Since $\norm{\bB} = \sqrt{M}$, the vector $\bp_m$
  can be bounded by
  \begin{equation}
    \label{eq:second_part1}
    \begin{split}
      \norm{\bp_m(t)} & = \norm{ -\frac{\imag}{k} \sum_{l=1}^{m-1}
        \frac{1}{l}\hat{f}^{(l)}_0(t) {\bB} {\hbbf}^{(m-l)}(t)}
      \leqslant \frac{\sqrt{M}}{k}\sum_{l=1}^{m-1}\frac{1}{l}
      \norm{\hbbf^{(l)}(t)}\norm{\hbbf^{(m-l)}(t)} \\
      & \leqslant \frac{\sqrt{M}}{k}\sum_{l=1}^{m-1}\frac{1}{l}
      K^{(l)}K^{(m-l)} \left(t^M+1\right)^2\exp\left(-2\lambda^{(1)}t\right),
    \end{split}
  \end{equation}
  where the last inequality is an application of the inductive
  hypothesis. Now we can take the norm on both sides of
  \eqref{eq:sol_sec} and plug in the inequalities \eqref{eq:exp_G} and
  \eqref{eq:second_part1} to get
  \begin{equation}
    \label{eq:norm_f2}
    \begin{split}
      \left\| {\hbbf}^{(m)}(t)\right\|_2 & 
        \leqslant \int_0^t \norm{\exp((t-s)\bG_m)} \norm{\bp_m(s)} \dd s \\
      & \leqslant C^{(5)} \sum\limits_{l=1}^{m-1}\frac{1}{l}K^{(l)}K^{(m-l)}
        \left(t^M+1\right) \exp\left(-\lambda^{(1)}t\right),
    \end{split}
  \end{equation}
  where
  \begin{equation}
    C^{(5)} = \frac{\sqrt{M}}{k} C^{(4)} \int_0^\infty
      \left(s^M+1\right)^2\exp\left(-\lambda^{(1)}s\right) \dd s.
  \end{equation}
  The inequality \eqref{eq:norm_f2} shows that \eqref{eq:fm_est} holds
  for
  \begin{equation}
    \label{eq:coe_m}
    K^{(m)} = C^{(5)}\sum\limits_{l=1}^{m-1}\frac{1}{l}K^{(l)}K^{(m-l)}
      = C^{(5)}\sum\limits_{l=0}^{m-2}\frac{1}{l+1} K^{(l+1)}K^{(m-l-1)}.
  \end{equation}
  Therefore for all positive integer $m$, the estimate
  \eqref{eq:fm_est} holds according to the principle of the mathematical
  induction.
\end{proof}

From the above theorem, we see that if the distribution function
\eqref{eq:ce_f} is approximated by a truncation of the series, then
such a finite series converges to a constant as $t \rightarrow
\infty$. To prove such a property for the infinite series
\eqref{eq:ce_f}, we still need to study the magnitude of each
coefficient $K^{(m)}$. The recursion relation \eqref{eq:coe_m} reminds
us of Jonah's theorem introduced in \cite{Hilton1990Ballot}:
\begin{lemma}[Jonah]
If $m_1 \geqslant 2m_2$, then
\begin{equation}
  \label{eq:Jonah}
  \left( \begin{array}{c} m_1+1 \\ m_2 \end{array} \right) =
    \sum_{l=0}^{m_2} \frac{1}{l+1}
      \left( \begin{array}{c} 2l \\ l \end{array} \right)
      \left( \begin{array}{c} m_1-2l \\ m_2-l \end{array} \right).
  \end{equation}
\end{lemma}
By this lemma, a general formula of $K^{(m)}$ can be explicitly
written, and then it can be properly bounded. The details are listed
in the following proof:

{\renewcommand\proofname{Proof of Theorem \ref{thm:vp}}
\begin{proof}
We first claim that
\begin{equation}
  \label{eq:K}
  K^{(m+1)} = \frac{1}{2^m}
    \left( \begin{array}{c} 2m \\ m \end{array} \right)
    K^{(1)} \left( C^{(5)} K^{(1)} \right)^m.
\end{equation}
Apparently the above equality holds for $m = 0$. If $m > 0$, we just
need to verify that \eqref{eq:K} fulfills \eqref{eq:coe_m}. By
inserting \eqref{eq:K} into \eqref{eq:coe_m} and cancelling out some
constants on both sides, we obtain
\begin{equation}
  \label{eq:rec}
  \frac{1}{2} \left( \begin{array}{c} 2m \\ m \end{array} \right) =
    \sum_{l=0}^{m-1} \frac{1}{l+1}
      \left( \begin{array}{c} 2l \\ l \end{array} \right)
      \left( \begin{array}{c} 2(m-l-1) \\ m-l-1 \end{array} \right).
\end{equation}
To verify this equality, we apply Jonah's theorem \eqref{eq:Jonah} and
let $m_1 = 2(m-1)$, $m_2 = m-1$. Thus the right hand side of
\eqref{eq:rec} matches the right hand side of \eqref{eq:Jonah}. The
left hand sides are also equal since
\begin{equation}
  \frac{1}{2} \left( \begin{array}{c} 2m \\ m \end{array} \right) =
    \frac{1}{2} \frac{(2m)!}{m!m!} =
    \frac{1}{2} \frac{2m}{m} \frac{(2m-1)!}{m!(m-1)!} =
    \left( \begin{array}{c} 2m-1 \\ m-1 \end{array} \right) =
    \left( \begin{array}{c} m_1+1 \\ m_2 \end{array} \right).
\end{equation}

Based on \eqref{eq:K}, it is easy to bound $K^{(m)}$ by
\begin{equation}
  K^{(m+1)} = 2^m \frac{(2m-1)!!}{(2m)!!} K^{(1)}
    \left( C^{(5)} K^{(1)} \right)^m \leqslant
      K^{(1)} \left( 2 C^{(5)} K^{(1)} \right)^m.
\end{equation}
Now we can use \eqref{eq:ce_f} to get
\begin{equation}
  \begin{split}
  \left\| \bbf(\cdot, t) - \hbbf^{(0)} \right\|_2 & \leqslant
    \sqrt{D} \sum_{m=0}^{+\infty} \epsilon^{m+1}
      \left\| \hbbf^{(m+1)} \right\|_2 \leqslant
    \sqrt{D} \sum_{m=0}^{+\infty} \epsilon^{m+1}
      K^{(m+1)} (t^M + 1) \exp(-\lambda^{(1)}t) \\
  & \leqslant \epsilon D K^{(1)}
    \sum_{m=0}^{+\infty} \left( 2\epsilon C^{(5)} K^{(1)} \right)^m
      (t^M + 1) \exp(-\lambda^{(1)} t).
  \end{split}
\end{equation}
Now it is clear that when $\epsilon < \left[ 2 C^{(5)} K^{(1)}
\right]^{-1}$, the inequality \eqref{eq:sol_vp} holds for
\begin{equation}
  C^{(3)} = \frac{\epsilon D K^{(1)}}{1 - 2\epsilon C^{(5)} K^{(1)}}.
  \qedhere
\end{equation}
\end{proof}}


\subsection{Analysis on the damping rate of the electric energy}
In this section, we will analyze how the filter affects the damping
rate and the oscillation frequency of the electric energy in the
problem of Landau damping. Below we will first review the analysis on
the original VP equation, and then the filtered equation will be
discussed.

\subsubsection{Analysis on the VP equation}
In the linear Landau damping problem, the plasma wave has very small
amplitude. Thus we can assume that
\begin{equation}
  \label{eq:linearize_dis}
  f(x,\xi,t) = f^{(0)}(\xi) + \epsilon f^{(1)}(x,\xi,t),
\end{equation}
where $f^{(0)}(\xi)$ is the background distribution, which is uniform
in $x$, and $f^{(1)}(x,\xi)$ has the order of magnitude $O(1)$.
Inserting this equation into the VP equation and ignoring the terms
quadratic in $\epsilon$, we get the following linearized VP equation:
\begin{equation}
  \label{eq:linearize_vlasov}
  \pd{f^{(1)}}{t} + \xi\pd{f^{(1)}}{x} + E^{(1)}\pd{f^{(0)}}{\xi} = 0,
    \qquad \pd{E^{(1)}}{x} = \int_{\bbR} f^{(1)} \dd \xi.
\end{equation}
If we further assume that the plasma wave being considered takes the
form of a plane wave traveling in the $x$-direction:
\begin{equation}
  f^{(1)}(x, t, \xi) =  \hat{f}^{(1)}(\xi)\exp(-\imag\omega t+\imag k x),
  \qquad E^{(1)}(x,t) = \hat{E}^{(1)} \exp(-\imag \omega t + \imag k x),
\end{equation}
we obtain the dispersion relation
\begin{equation}
  \label{eq:dispersion_vlasov}
  \frac{1}{k^2}\int_{\bbR}\frac{1}{\xi - \omega / k}
    \pd{f^{(0)}}{\xi} \dd \xi = 1.
\end{equation}
Let $\omega = \omega_{p} + \imag \gamma$ be the solution to
\eqref{eq:dispersion_vlasov}. Then $\gamma$ and $\omega_p$ are
respectively the damping rate and the oscillation frequency of the
electric energy. We refer the readers to
\cite{Goldston2016Introduction} for the more details on the dispersion
relation.

\subsubsection{Analysis on the filtered VP equation}
In order to take into account the filters, we assume that the filter
\eqref{eq:sigma_M} is used and consider the ``modified equation''
\eqref{eq:filtered_vlasov}, so that a similar analysis can be carried
out. Here we also assume that the distribution function is a small
perturbation of the background distribution \eqref{eq:linearize_dis},
and the linearized VP equation with filter can be written as
\begin{equation}
  \label{eq:linear_vlasov}
  \pd{f^{(1)}}{t} + \xi \pd{f^{(1)}}{x}
    + E^{(1)} \pd{f^{(0)}}{\xi} = \mH f^{(1)}, \qquad
  \pd{E^{(1)}}{x} = \int_{\bbR} f^{(1)} \dd \xi,
\end{equation}
where the operator $\mathcal{H}$ is defined by $\mH = (-1)^{p+1}
\alpha (\Delta t \, M^p)^{-1} \mathcal{D}^p$, and we have assumed $\mH
f^{(0)} = 0$, which holds for the initial condition
\eqref{eq:initial_vlasov}. Let
\begin{equation} \label{eq:def_g}
  g(x,\xi,t) = \exp(-t\mH)f^{(1)}(x,\xi,t).
\end{equation}
Then the function $g$ formally satisfies
\begin{equation}
  \label{eq:g}
  \exp(t\mH) \pd{g}{t} =
    -\left(\xi \exp(t\mH)\pd{g}{x} + E^{(1)} \pd{f^{(0)}}{\xi}\right),
  \qquad \pd{E^{(1)}}{x} =
    \int_{\bbR} \exp(t \mH) g \dd \xi.
\end{equation}
Now we assume that the equation \eqref{eq:g} has a plane-wave solution:
\begin{equation}
  \label{eq:fourier_trans}
  g(x,\xi,t) = \hat{g}(\xi) \exp(-\imag \omega t + \imag k x),
  \qquad E^{(1)}(x, t) = \hat{E}^{(1)} \exp(-\imag \omega t + \imag k x),
\end{equation}
which changes \eqref{eq:g} into
\begin{equation}
  \label{eq:dis_f1}
  \imag (\xi k - \omega) \exp(t \mH) \hat{g}
    = -\hat{E}^{(1)} \pd{f^{(0)}}{\xi}, \qquad
  \imag k \hat{E}^{(1)} = \int_{\bbR} \exp(t \mH) \hat{g} \dd \xi.
\end{equation}
By dividing the first equation by $\xi k - \omega$ and integrating
with respect to $\xi$, the amplitude of the electric field
$\hat{E}^{(1)}$ can be cancelled out and we get the result
\begin{equation}
k = \int_{\bbR} \frac{1}{\xi k - \omega} \pd{f^{(0)}}{\xi} \dd \xi,
\end{equation}
which turns out to be exactly the same as the result of the original
VP equation \eqref{eq:dispersion_vlasov}. Such analysis shows that the
modified equation and the original VP equation have the same damping
rate and oscillation frequency for the electric field. Therefore
adding filters does not ruin the major behavior of the linear Laudau
damping.

\begin{remark}
The above analysis is only formally correct, since in
\eqref{eq:def_g}, the operator $\exp(-t \mH)$ is generally an
unbounded operator. Therefore it still remains to prove that $f^{(1)}$
lies in the domain of this operator for all $x$ and $t$, which is not
yet done. Currently we are not too concerned about this point since
the main idea of this section is to provide a clue for the reliability
of filters.
\end{remark}


\section{Numerical Results}
\label{sec:num}

In order to verify the theoretical results and show how the recurrence
is suppressed by the filter, the numerical experiments for both the
advection equation and the Vlasov equation are carried out. In the
literature, the recurrence for the Landau damping problem is usually
observed from the evolution of the self-consistent electric energy
$\mE(t)$ defined in \eqref{eq:electric_energy}. Therefore in all the
numerical tests, we are going to use the same quantity
\eqref{eq:electric_energy} to show the effect of the filter. Note that
for the advection equations considered in Section \ref{sec:con},
we do not have $E(x,t) = E_{\mathrm{sc}}(x,t)$ as in the Vlasov
equation, and thus the quantity $\mE(t)$ is merely a functional of the
distribution function $f(x,\xi,t)$ defined by \eqref{eq:normal_force}
\eqref{eq:electric_energy}, and does not appear explicitly in the
transport equation.

The filter we adopt in our numerical tests is the one proposed in
\cite{HouLi2007}, which is identical to the filter defined in
\eqref{eq:sigma_M} with $\alpha = 36$ and $p = 36$. The corresponding
filter matrix $\bH$ can be written as \eqref{eq:h} with
\begin{equation}
  h_i = -36 \Delta t^{-1}(i/M)^{36}, \quad i = 0, \cdots, M.
\end{equation}
As mentioned in Section \ref{sec:Hermite}, the time step $\Delta t$ is
set as $\Delta t = C / \sqrt{M}$, and the constant $C$ is chosen as
$0.5$ in all the tests. Since it is much easier to obtain
$\mE(t)$ from the Fourier coefficients $\hbbf^{(m)}$, we solve
\eqref{eq:fourier_system_vlasov} instead of the original transport
equation \eqref{eq:vlasov} in all our experiments.

\subsection{Advection equation}
In this section, we focus on the advection equation
\eqref{eq:transport} with initial condition \eqref{eq:init}. This
initial condition gives the following Fourier coefficients:
\begin{equation*}
  \hbbf^{(0)}(0) = (1, 0, \cdots, 0)^T, \quad
  \hbbf^{(-1)}(0) = \hbbf^{(1)}(0)= (\epsilon/2, 0, \cdots, 0)^T, \quad
  \hbbf^{(m)}(t) = 0, \quad |m| \geqslant 2.
\end{equation*}
In this case, the equations \eqref{eq:fourier_system_te} shows that
each $\hbbf^{(m)}$ can be solved independently, and when $|m|
\geqslant 2$, the initial condition shows that $\hbbf^{(m)}(t) \equiv
0$. Consequently, the function $\mE(t)$ has a simple form
\begin{equation}
  \label{eq:Ef_norm1}
  \mE(t) =\sqrt{\frac{D}{k^2} \left(
    \left|\hat{f}_0^{(-1)}(t)\right|^2 +
    \left|\hat{f}_0^{(1)}(t)\right|^2
  \right)}.
\end{equation}
The exact solution of $\mE(t)$ can be directly obtained from
\eqref{eq:exact} as
\begin{equation}
  \label{eq:analy_E}
  \mE(t) = \frac{\epsilon}{k} \sqrt{\frac{D}{2}} \exp(-k^2 t^2/2). 
\end{equation}
To obtain $\mE(t)$ with the spectral method, we just need to apply
\eqref{eq:fourier_system_sol2} for $m = \pm 1$, which requires
diagonalization of matrices $\bA_{\pm 1}$.

In this example, we set $D = 4 \pi$ and $k = 2\pi / D = 1/2$. We
first verify that all the eigenvalues of $\bA_{\pm 1}$ have negative
real parts. Since $\bA_1 = \bA_{-1}^*$, the eigenvalues of $\bA_1$ are
the complex conjugates of those of $\bA_{-1}$. Therefore it is
sufficient to just look at $\bA_1$. The distribution of the
eigenvalues of $\bA_1$ on the complex plane is plotted in Figure
\ref{fig:ex1_eig}, and the eigenvalues of $\imag k \bA$ (the case
without filter) are also given as a reference. As predicted, all the
eigenvalues of $\imag k \bA$ locate exactly on the imaginary axis,
while all the eigenvalues of $\bA_1$ have strictly negative real
parts.

\begin{figure}[!ht]
  \centering
  \subfigure[$M=30$]{
    \begin{overpic}[height=5cm]{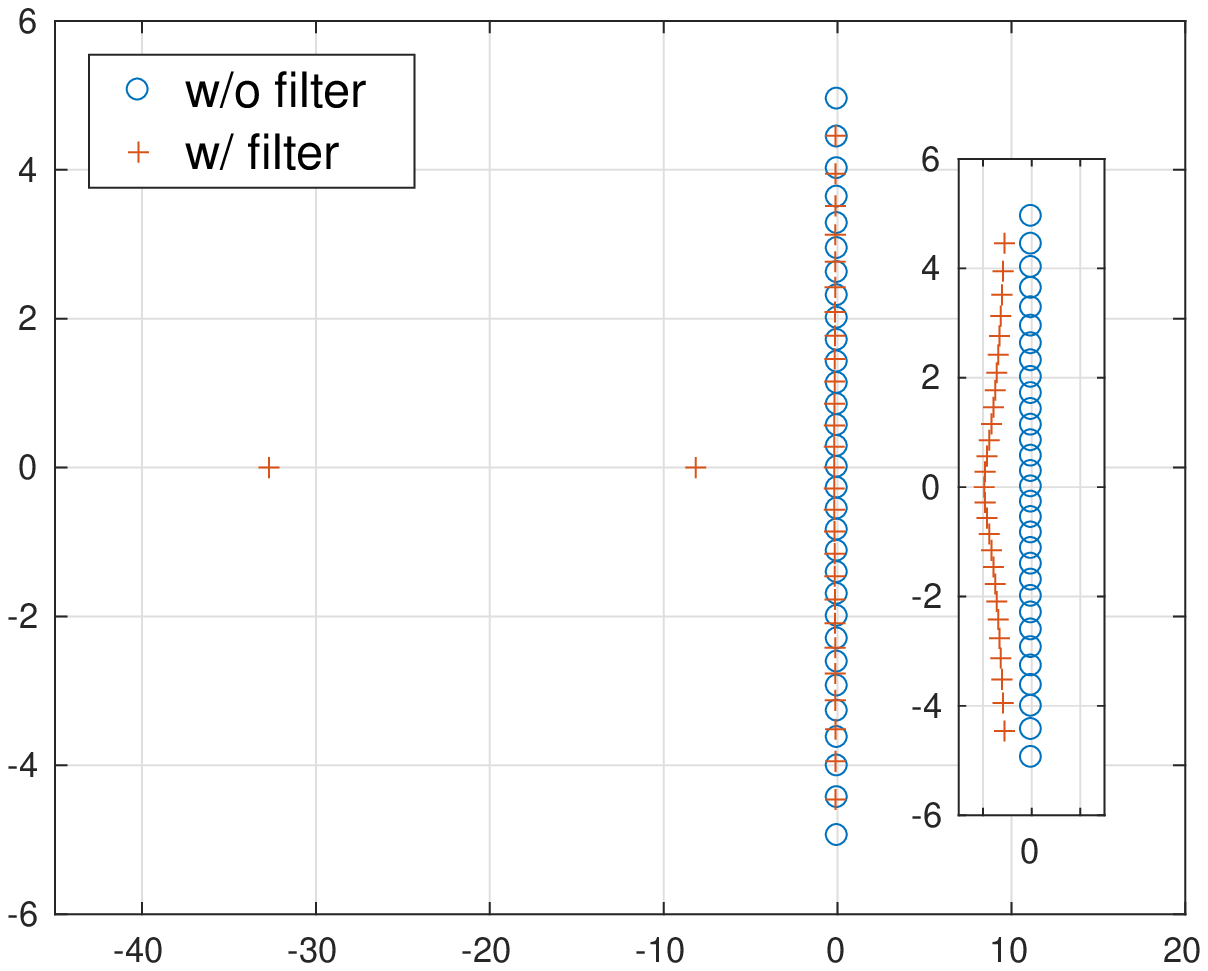}
    \end{overpic}
  }\quad
  \subfigure[$M=60$]{
    \begin{overpic}[height=5cm,clip]{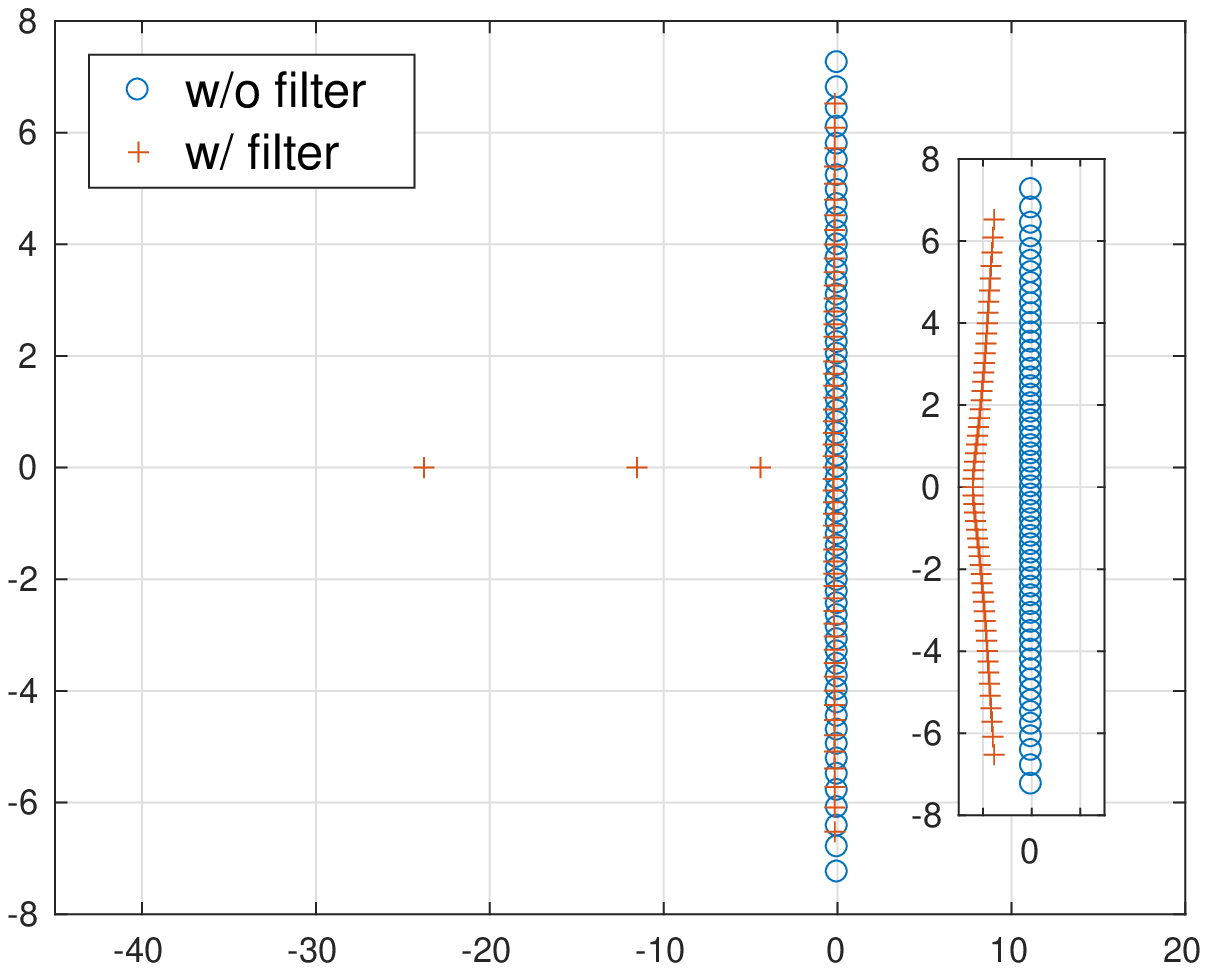}
    \end{overpic}
  }
  \caption{The eigenvalues of $\imag k \bA$ and $\bA_1$ with
    $M=30$ and $60$. The circles are the eigenvalues of
    $\imag k \bA$ and the plus signs are the eigenvalues of $\bA_1$.}
  \label{fig:ex1_eig}
\end{figure}

The time evolutions of $\mE$ and its logarithm are plotted
respectively in Figure \ref{fig:ex1_sol} and \ref{fig:ex1_logsol}.
Initially, both numerical solutions with and without the filter decay
in the same way as the exact solution. When the evolution reaches the
first ``critical'' point, almost simultaneously, both solutions
leave the exact solution and exhibit a fast growth. For the case
without filter, the value of $\mE(t)$ bounces almost back to its
initial value before the next decay, and such a process repeats
periodically, which implies a constantly appearing recurrence. In
general, larger number of moments leads to a longer recurrence time,
which can also be observed by comparing the results with $M = 30$ and
$M = 60$. When the filter is applied, the solutions behave similarly
except that the peak values get smaller as $t$ gets larger, due to the
damping effect of the filter. It can also be observed that the damping
rate for $M = 60$ is higher than that for $M = 30$, which indicates
that for the same filter, systems with more moments may have better
suppression of recurrence. We comment here that the recurrence is not
completely eliminated since $\mE(t)$ is still not monotonically
decreasing.  Nevertheless, from Figure \ref{fig:ex1_sol}, one sees
that the difference between the numerical solution and the analytical
solution is negligible when $t$ is large.

\begin{figure}[!ht]
  \centering
  \subfigure[$M=30$]{
    \begin{overpic}[height=5cm]{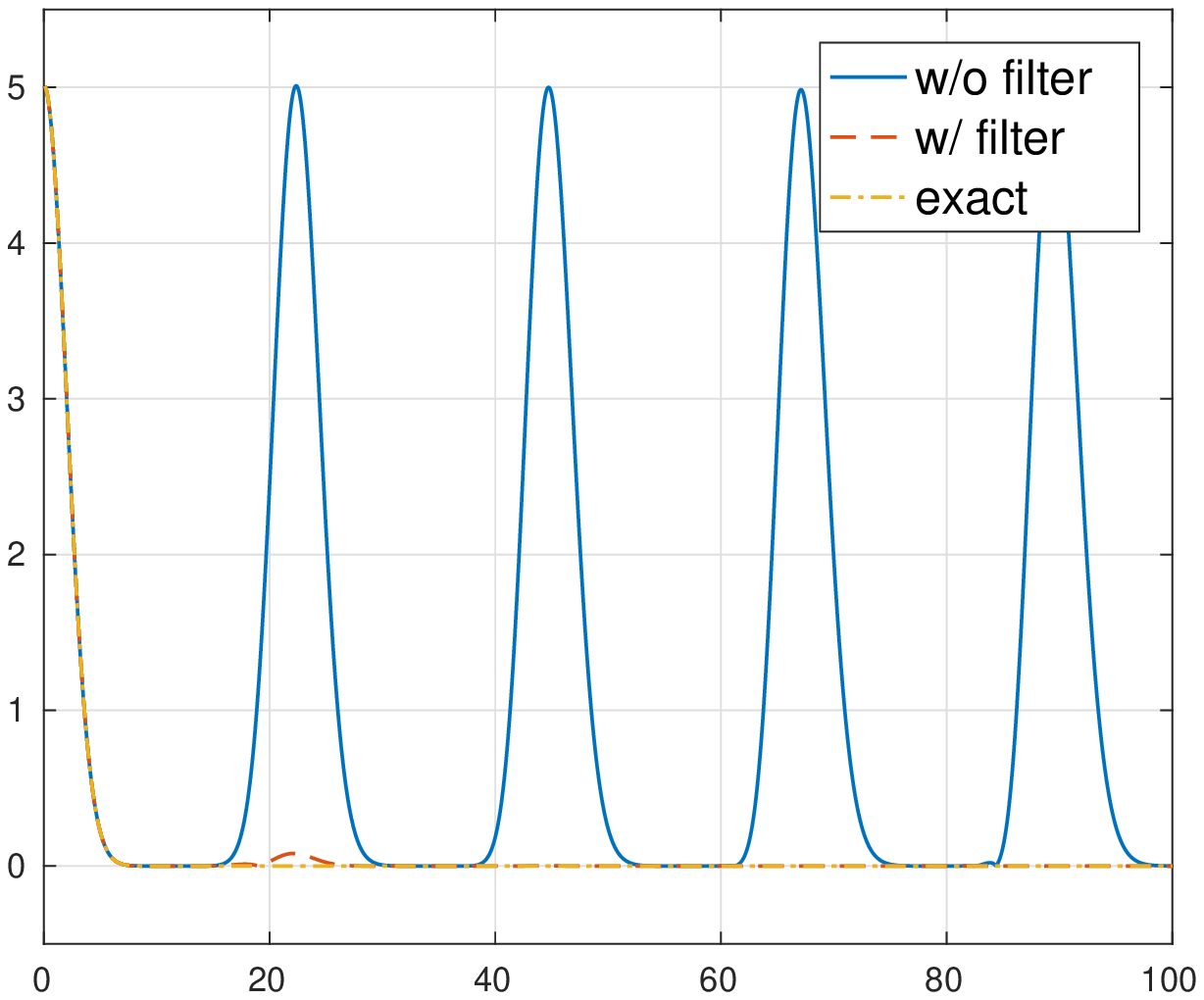}
    \end{overpic}
  } \quad
  \subfigure[$M=60$]{
    \begin{overpic}[height=5cm]{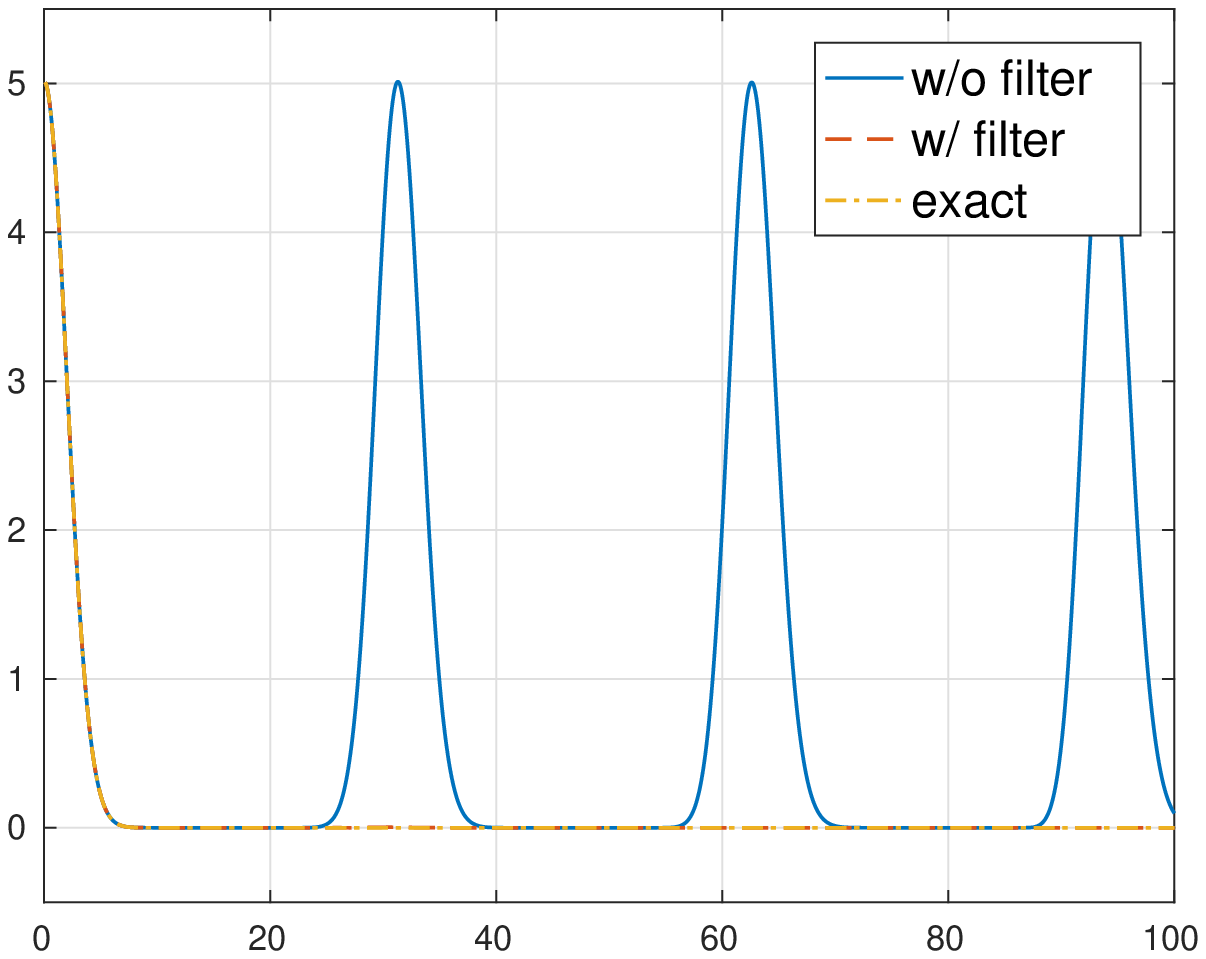}
    \end{overpic}
  }
  \caption{The time evolution of $\mE(t)/\epsilon$ with $M=30$ and $60$.}
  \label{fig:ex1_sol}
\end{figure}

\begin{figure}[!ht]
  \centering
  \subfigure[$M=30$]{
    \begin{overpic}[height=5cm]{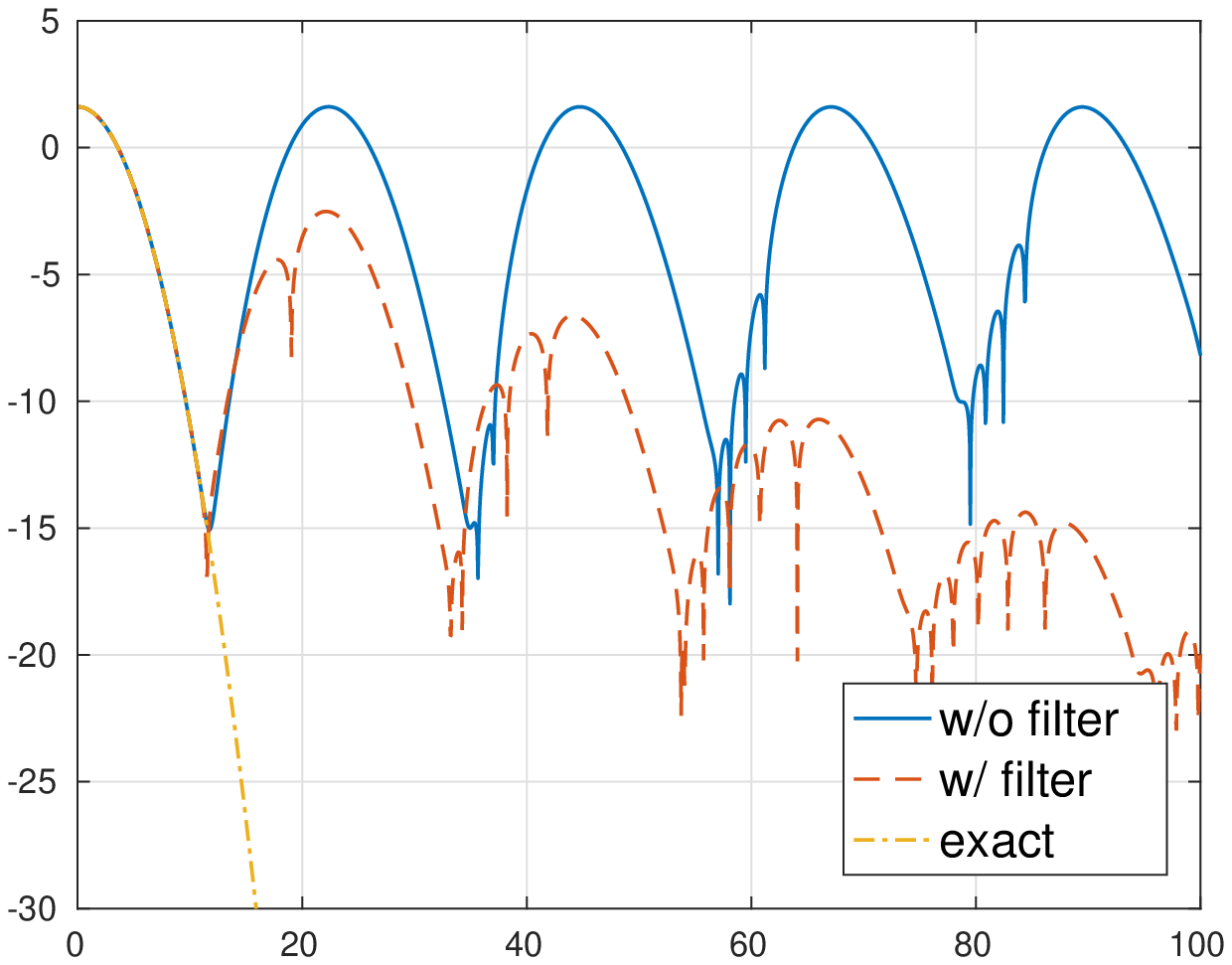}
    \end{overpic}
  } \quad
  \subfigure[$M=60$]{
    \begin{overpic}[height=5cm]{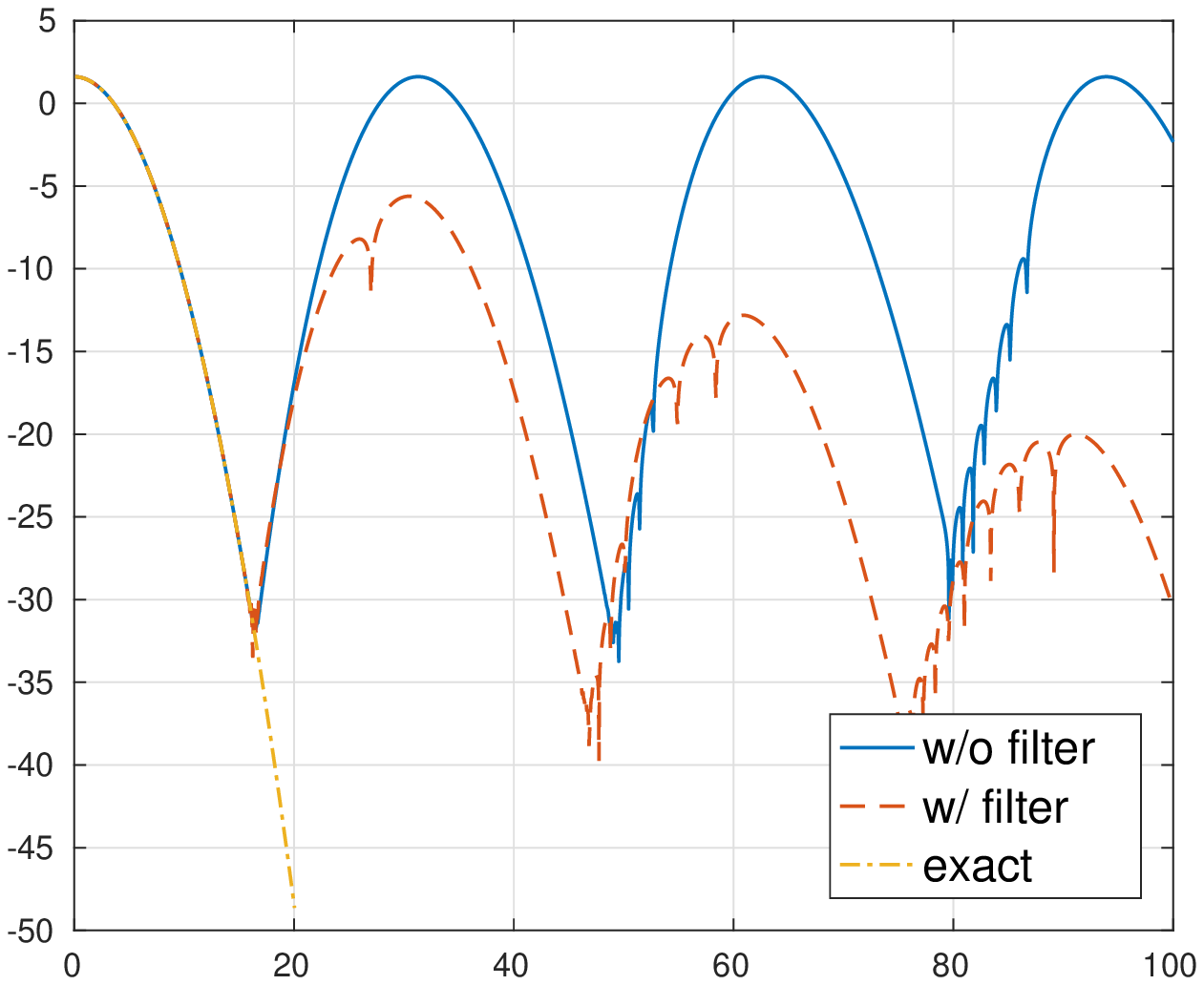}
    \end{overpic}
  } \quad 
  \caption{The time evolution of $\log (\mE(t)/\epsilon)$ with $M=30$ and $60$.}
  \label{fig:ex1_logsol}
\end{figure}

\subsection{Advection equation with given force field}
\label{sec:num_force}

\begin{figure}[!ht]
  \centering
  \psfrag{w/o filter m === 3}{\tiny{w/o filter $m_c=3$}}
  \psfrag{w/o filter m = 5}{\tiny{w/o filter $m_c=5$}}
  \psfrag{w/ filter m = 3}{\tiny{w/ filter $m_c=3$}}
  \psfrag{w/ filter m = 5}{\tiny{w/ filter $m_c=5$}}
  \subfigure[$M=30$]{
    \begin{overpic}[height=5cm]{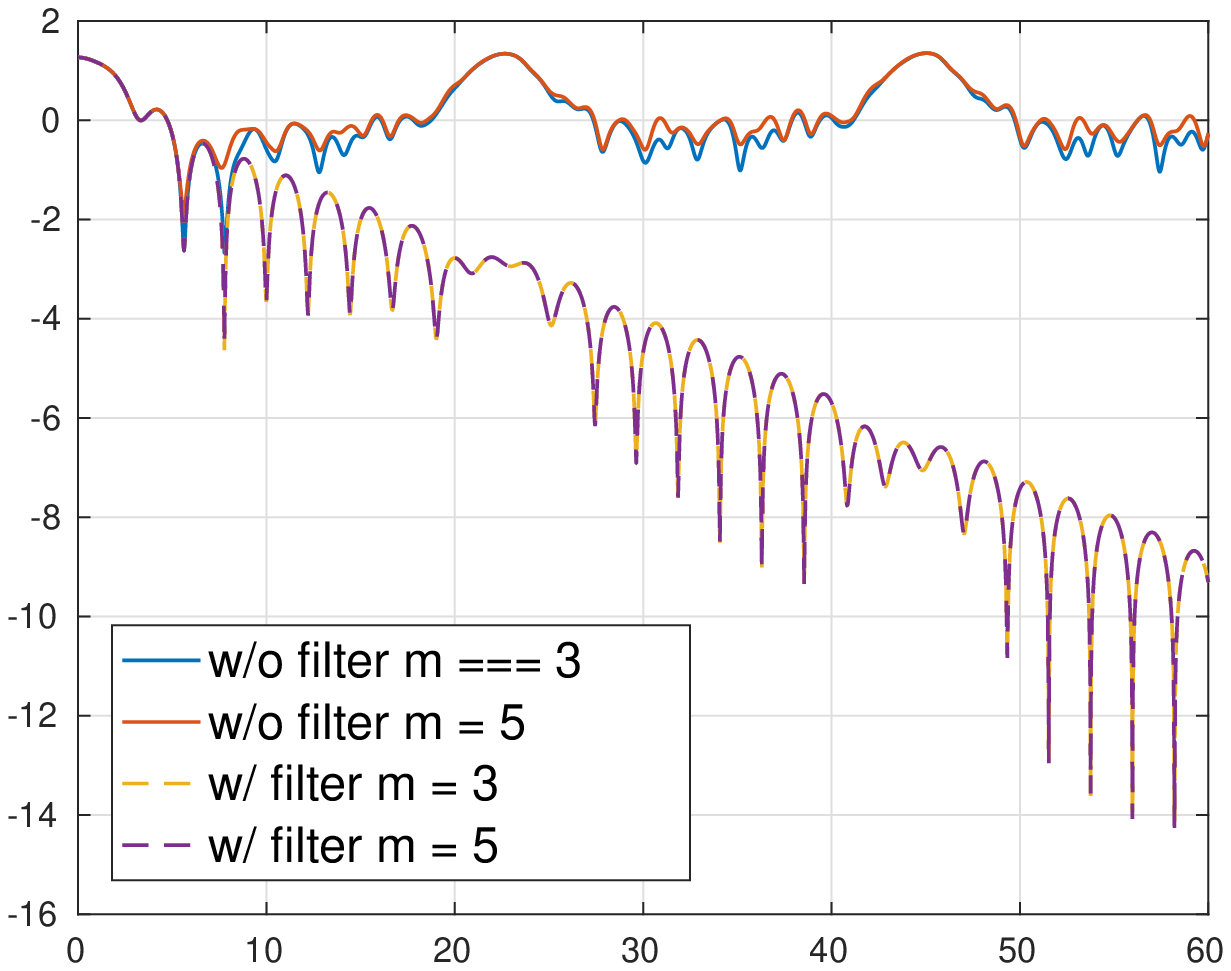}
    \end{overpic}
  } \quad \subfigure[$M=60$]{
    \begin{overpic}[height=5cm]{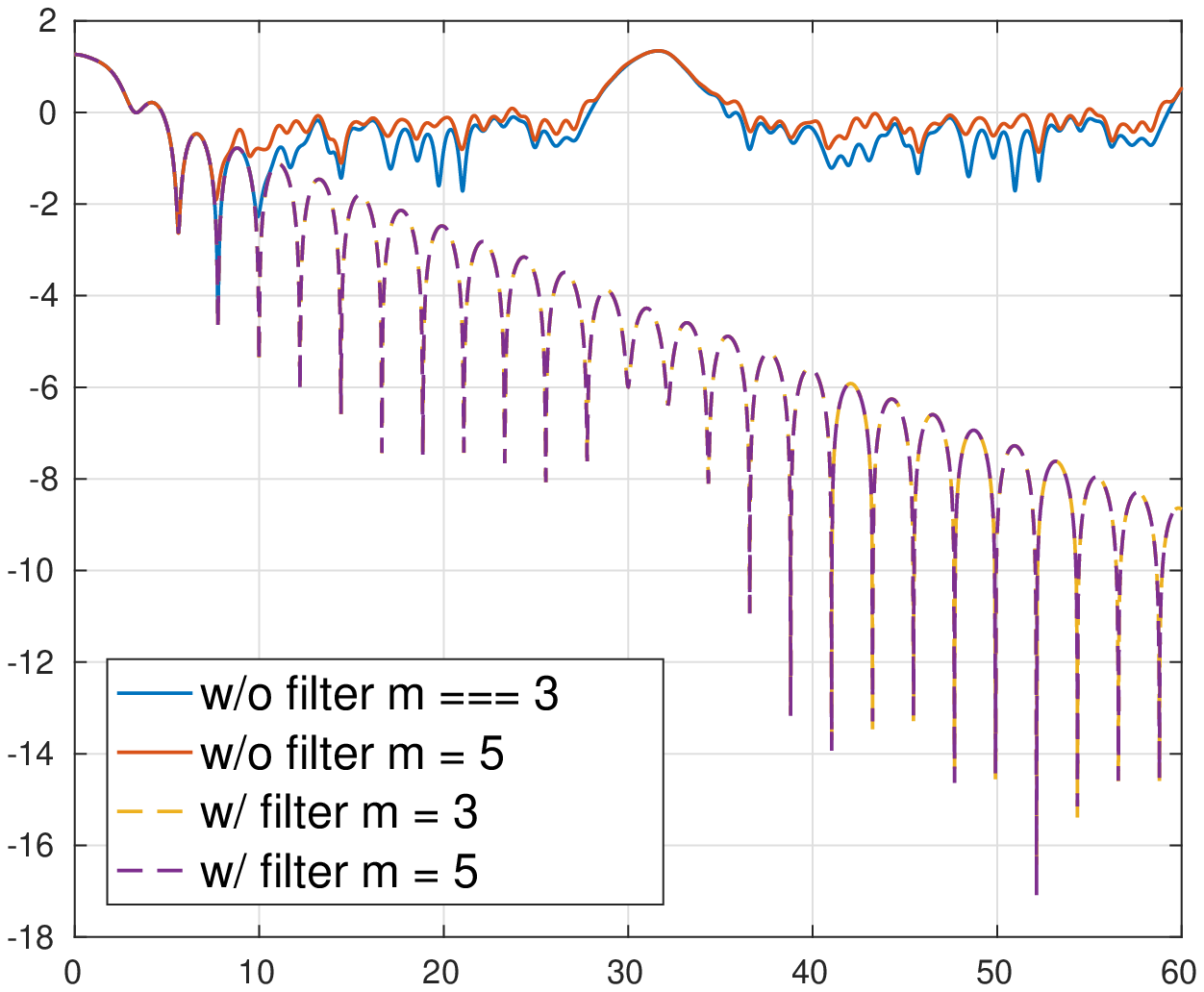}
    \end{overpic}
  }
  \caption{The evolution of $\log \mE(t)$ for the example in
    Section \ref{sec:num_force}.}
  \label{fig:ex2_logsol}
\end{figure}

\begin{figure}[!ht]
  \centering
  \psfrag{w/o filter m = 3}{\tiny{w/o filter $m_c=3$}}
  \psfrag{w/o filter m = 5}{\tiny{w/o filter $m_c=5$}}
  \psfrag{w/ filter m = 3}{\tiny{w/ filter $m_c=3$}}
  \psfrag{w/ filter m = 5}{\tiny{w/ filter $m_c=5$}}
  \psfrag{slope========}{\tiny{$y =-\lambda^{(1)}t$}}
  \subfigure[$M=30$]{
    \begin{overpic}[height=5cm]{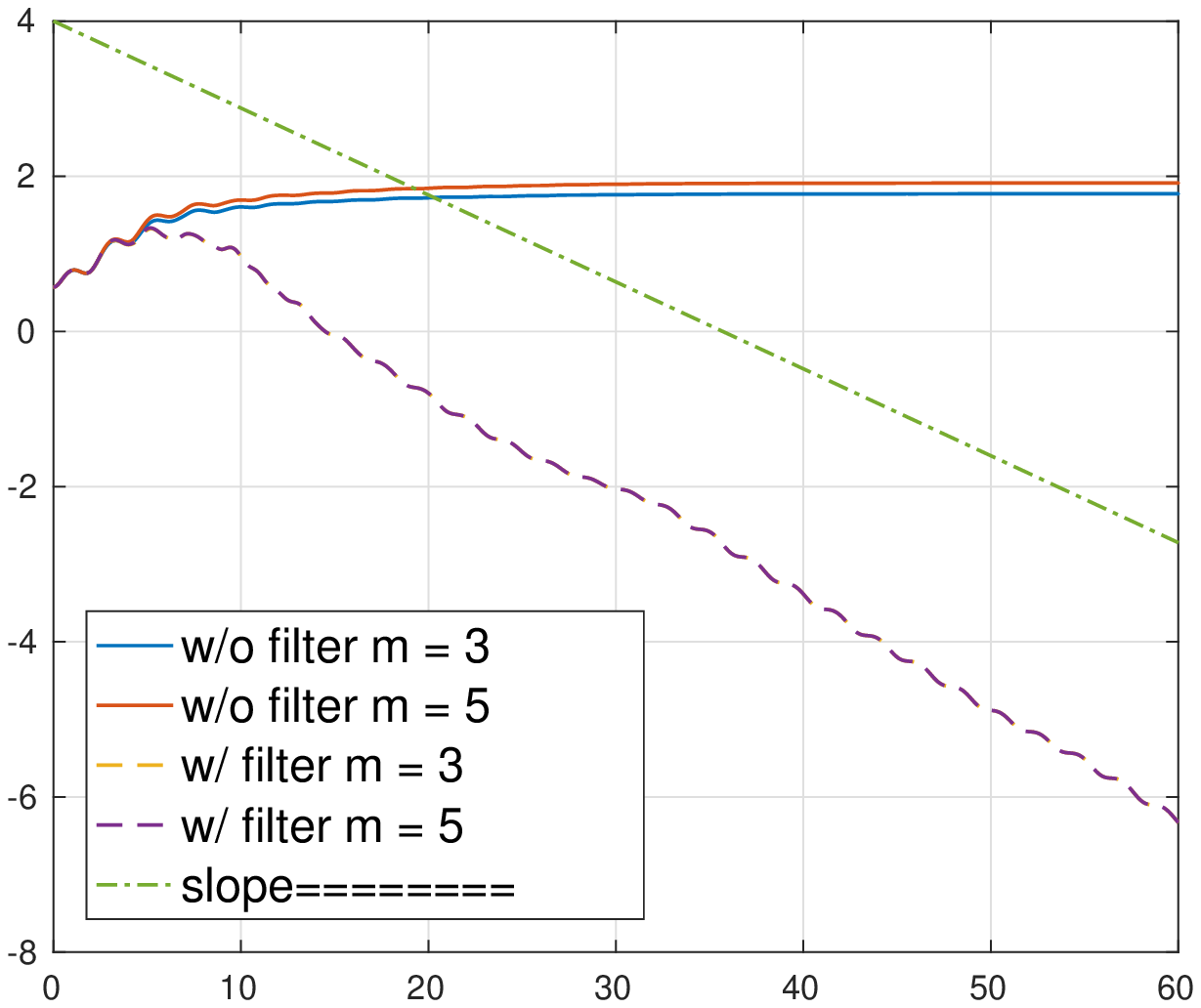}
    \end{overpic}
  } \quad \subfigure[$M=60$]{
    \begin{overpic}[height=5cm]{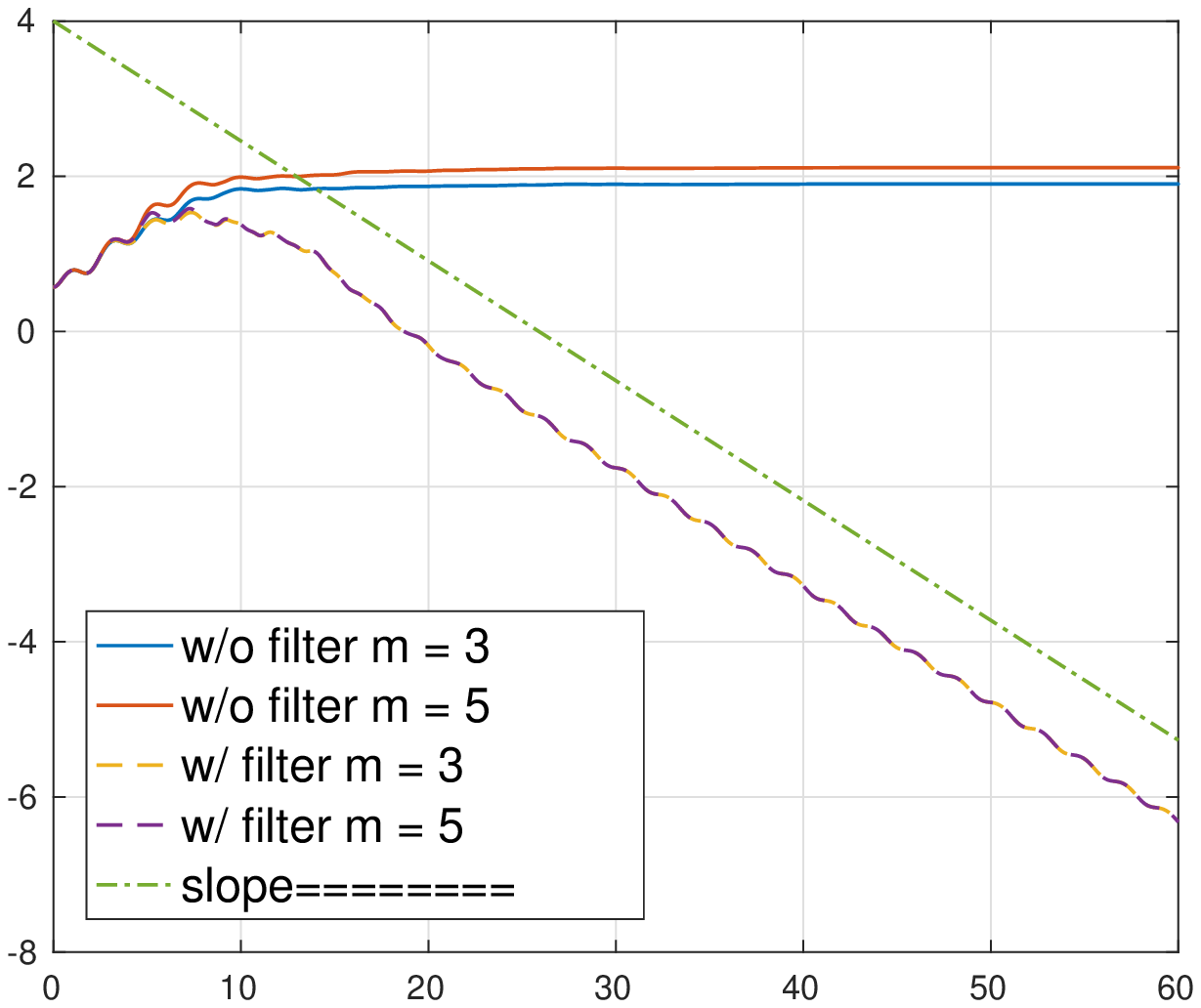}
    \end{overpic}
  }
  \caption{The evolution of $\log(\|\hbbf - \hbbf^{(0)}\|_2)$
    for the example in Section \ref{sec:num_force}.}
  \label{fig:ex2_loghbbf}
\end{figure}

\begin{figure}[!ht]
  \centering
  \psfrag{m = 1}{\tiny{$\log\|\hbbf^{(1)}\|_2$}}
  \psfrag{m = 2}{\tiny{$\log\|\hbbf^{(2)}\|_2$}}
  \psfrag{m = 3}{\tiny{$\log\|\hbbf^{(3)}\|_2$}}
  \psfrag{m = 4}{\tiny{$\log\|\hbbf^{(4)}\|_2$}}
  \psfrag{m = 5}{\tiny{$\log\|\hbbf^{(5)}\|_2$}}
  \psfrag{slope===}{\tiny{$y=-\lambda^{(1)}t$}}
  \subfigure[$M=30, m_c = 5$]{
    \begin{overpic}[height=5cm]{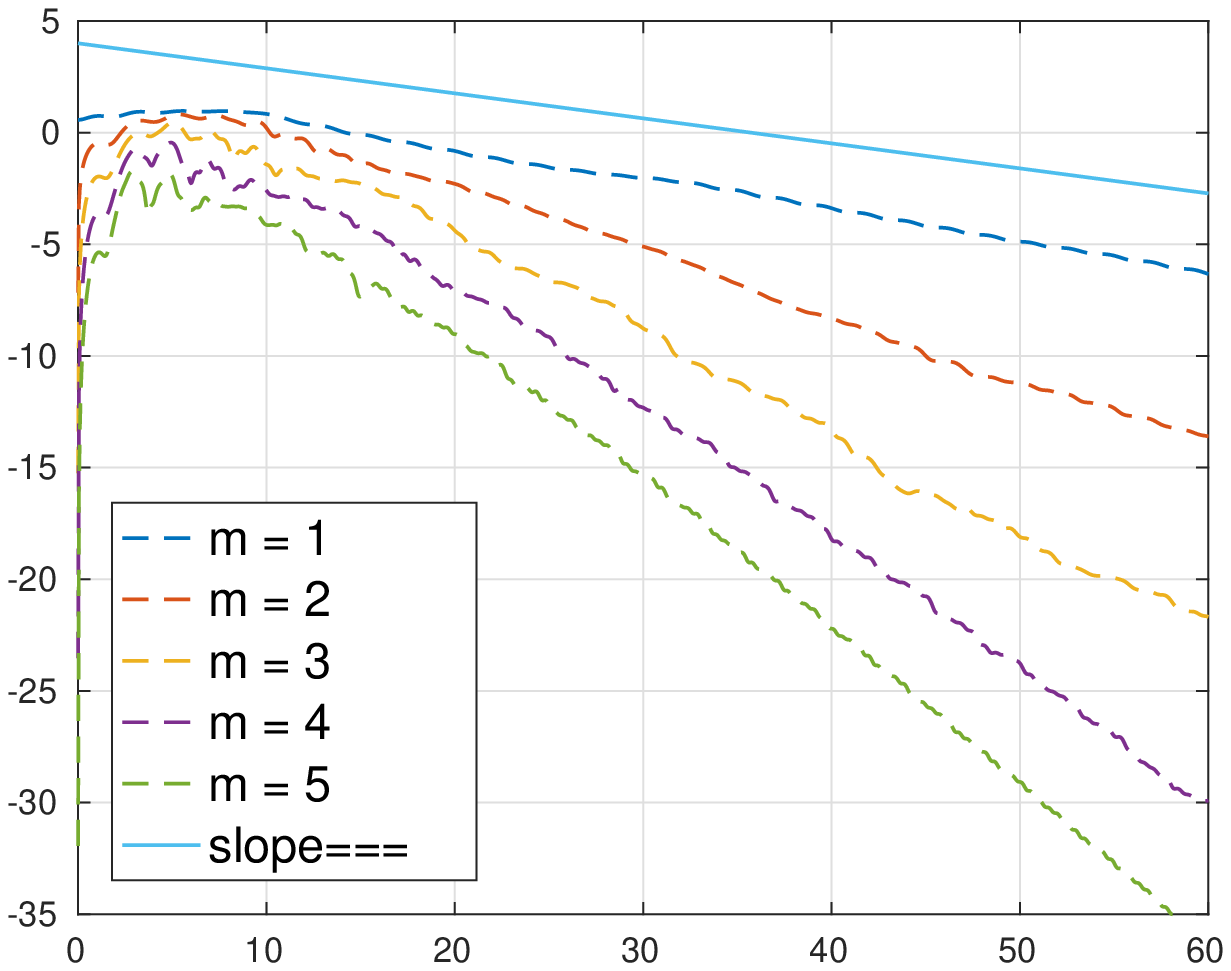}
    \end{overpic}
  }
  \quad 
  \subfigure[$M=60, m_c = 5$]{
    \begin{overpic}[height=5cm]{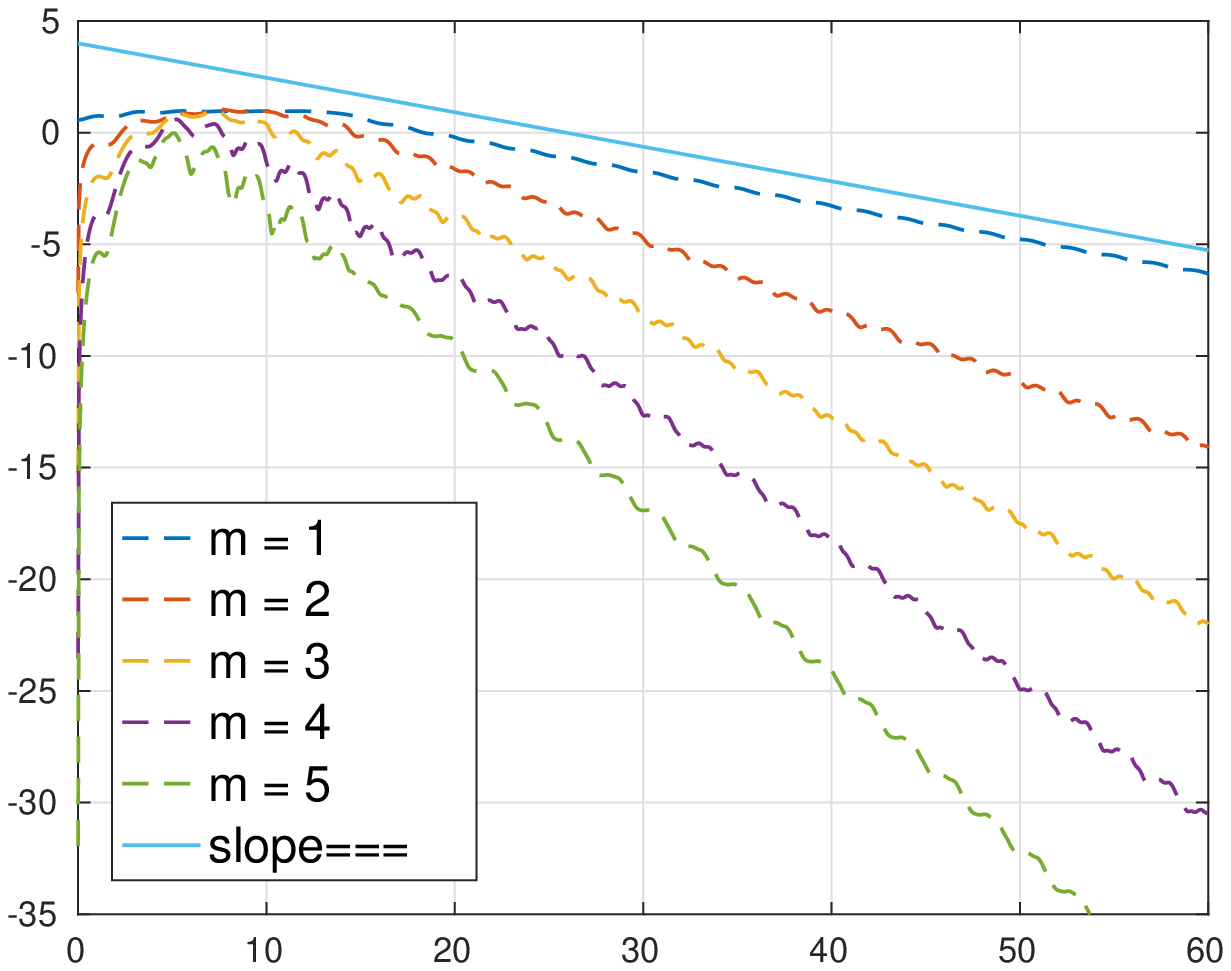}
    \end{overpic}
  }
  \caption{The evolution of $\log(\|\hbbf^{(m)}\|_2)$ for
    the example in Section \ref{sec:num_force}.}
  \label{fig:ex2_loghbbf_m}
\end{figure}

In this section, some numerical experiments for the transport equation
with a given force field \eqref{eq:electric_force} are performed. Here
we choose $\alpha(t)$ and $w(t)$ respectively as
\begin{equation}
  \alpha(t) = \gamma t, \qquad
  w(t) = \epsilon \cos(\omega t) \exp(\imag k x),
\end{equation}
so that it mimics the Landau damping process. The period and the wave
number are again chosen as $D = 4\pi$, $k = 0.5$. For Laudau damping,
these parameters give the decay rate $\gamma = 0.15336$ and the
plasma oscillation frequency $\omega = 1.416$ (see Section
\ref{sec:lld} for details), which are also adopted in definition of
$w(t)$. The initial condition is given by \eqref{eq:initial_vlasov}.
To see the contribution from waves with higher frequencies, we set
$\epsilon$ to be $0.9$. In order to carry out the numerical
simulation, only Fourier modes with $|m| \leqslant m_c$ are taken into
account.

Two numbers of moments $M = 30$ and $M = 60$ are simulated with $m_c =
3$ and $m_c = 5$. The time evolution of the logarithm of $\mE$ is
plotted in Figure \ref{fig:ex2_logsol}, which also shows clearly how
the filter improves the solution. Figure \ref{fig:ex2_loghbbf} also
shows such an effect from another point of view, where all the
non-constant modes are taken into account. For the case with filters,
the decay of each mode is given in Figure \ref{fig:ex2_loghbbf_m}. A
reference decay rate given by the largest real part of the eigenvalues
of $\bA_m$ is also provided as a reference. It is seen that the decay
rates of all the Fourier modes are controlled by this reference line.

\subsection{Linear Landau damping} \label{sec:lld}
The numerical effect of the filter in the simulation of linear Landau
damping has been studied in a number of previous works
\cite{parker2015fourier,pezzi2016collisional}. Here some simple
results are presented for our particular discretization
\eqref{eq:hermit_expansion}. The parameters are chosen as
$\epsilon = 0.001$, $D = 4 \pi$ and $ k = 1/2$. Since $\epsilon$ is
quite small, it is sufficient to consider the leading order term
$\epsilon \hat{E}^{(1)}(t) \exp(\imag k x)$ in the expansion
\eqref{eq:ce_E}.

The evolution of $\log \mE(t)$ with $M = 90$ and $M = 120$ is plotted
in Figure \ref{fig:ex3_logsol}, where the solutions with and without
filters are both given. We can find that the two solutions are almost
the same before the recurrence occurs. Similar to in the last
subsection, after the recurrence time, the results without filters
become unreliable, while the results with filters still decay
exponentially. To be precise, the numerical decay rate is obtained by
least-square fitting of peak value points before time $t_F$. Two
different values of $t_F$ (respectively before and after recurrence)
are used to fit the slope, and the results are compared with the
theoretical decay rate $\gamma = 0.15336$ (see e.g. \cite{Eric}).
Particularly, the second $t_F$ is chosen around twice the recurrence
time. Table \ref{tab:slope} shows that before recurrence, the decay
rates of both methods match very well with the theoretical result, and
after the recurrence time, an accurate decay rate can still be kept
for a long time by filtering. In this example, the filter improves the
numerical result dramatically, and such a method looks very promising
in the numerical simulation of plasmas.

\begin{figure}[!ht]
  \centering
  \subfigure[$M=90$]{
    \begin{overpic}[height=5cm]{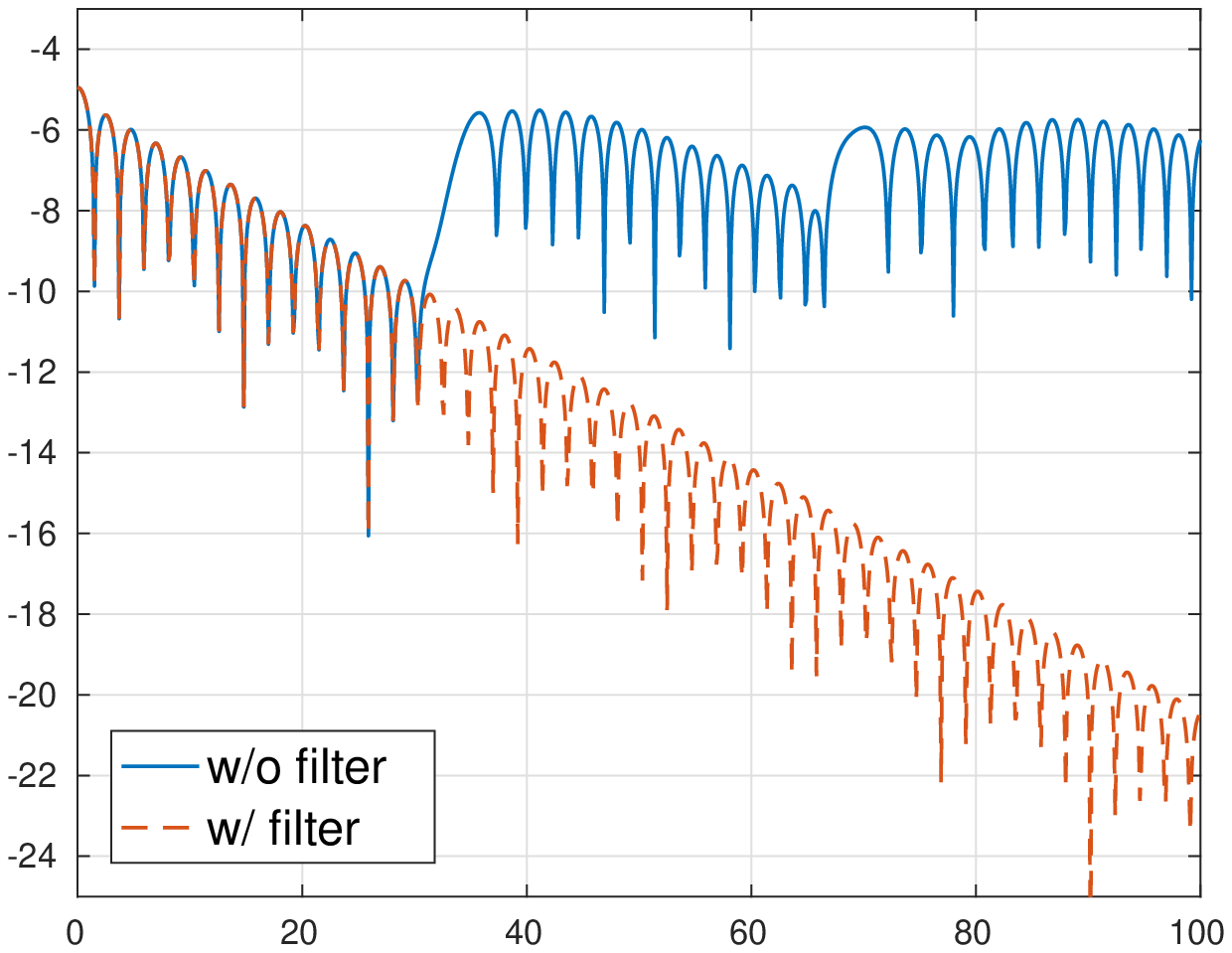}
    \end{overpic}
  }\quad
  \subfigure[$M=120$]{
    \begin{overpic}[height=5cm]{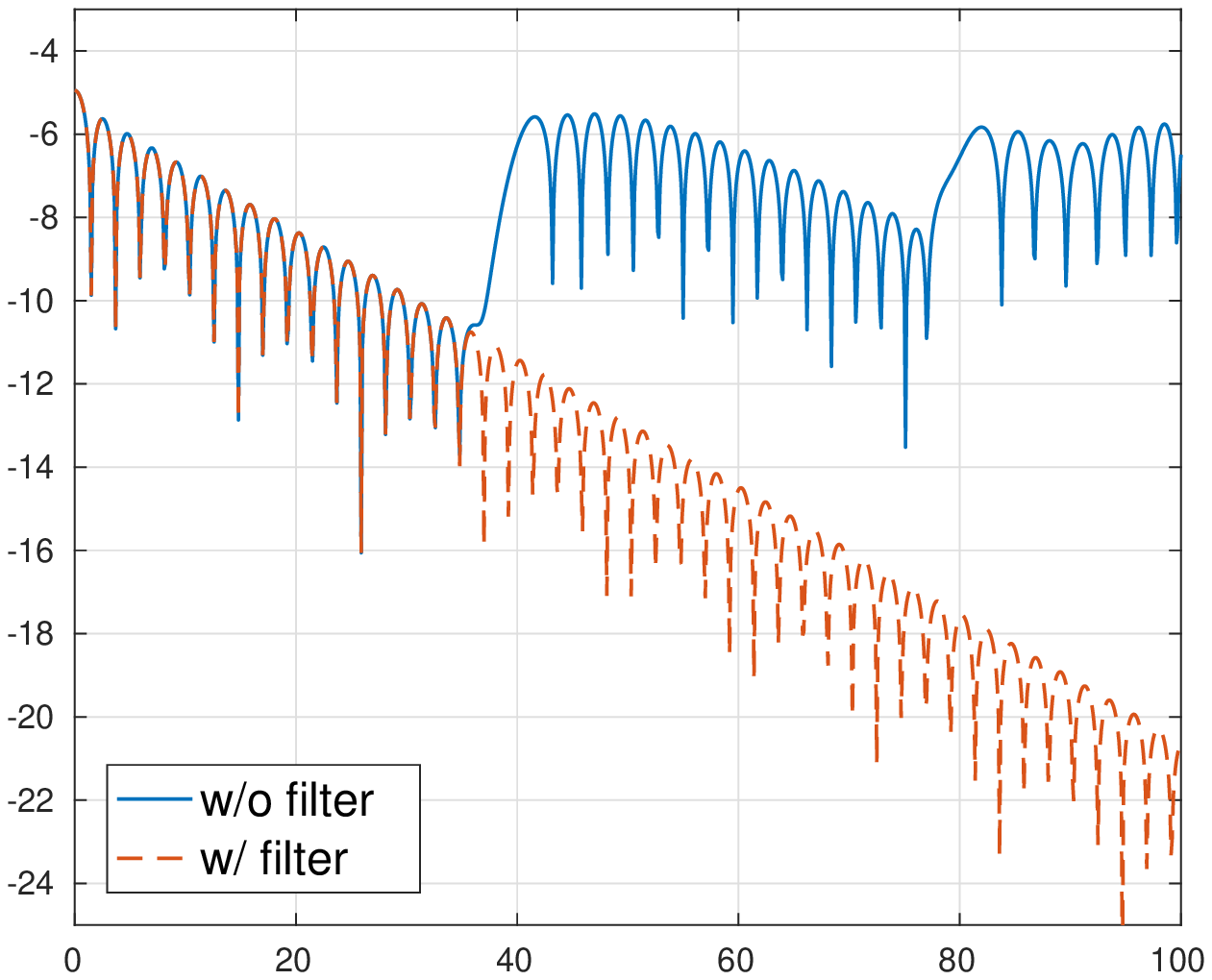}
    \end{overpic}
  } 
  \caption{The evolution of $\log \mE(t)$ for the Laudau damping problem.}
  \label{fig:ex3_logsol}
\end{figure}

\begin{table}[!ht]
  \centering
  \caption{The numerical decay rates of the electric energy}
  \label{tab:slope}
  \begin{tabular}{c|c|c|c||c|c|c|c}
    $M$ & w/o filter& w/ filter&  w/ filter  & $M$ & w/o filter& w/ filter&  w/ filter  \\
    \hline \hline
    $30$     &$t_F= 12$ &  $t_F = 12$& $t_F= 24$  &  $60$  &$t_F= 22$ &  $t_F = 22$& $t_F=44$  \\    
        & $0.155038$ & $0.1550545$ & $0.200223$ &    &  $0.1540681$ & $0.1540679$ & $0.166939$\\
    \hline
    $90$  & $t_F= 26$ & $t_F = 26$& $t_F=52$  &    $120$& $t_F= 30$ & $t_F = 30$& $t_F=60$  \\    
        & $0.154173$ & $0.154173$ & $0.152892$ &    & $0.153780$ & $0.153780$ & $0.153629$ \\
  \end{tabular}
\end{table}


\section{Conclusion} \label{sec:conclusion} 

In this paper, we have systematically analyzed the effect of the
filter on the numerical solutions to the transport equations using
Hermite-spectral method. The theoretical analysis on two types
of transport equations is proposed respectively. It is both
rigorously proven and numerically validated that the filter makes all
the non-constant modes damp exponentially, and therefore suppresses
the recurrence. The example of Landau damping shows that the filter
has only negligible effect on the damping rates of the electric
energy, which illustrates that adding filter is a promising method to
relieve the recurrence phenomenon. Analysis on the case with magnetic
field will be studied in the future.


\bibliographystyle{plain}
\bibliography{../article}
\end{document}